\documentclass[a4paper,11pt,leqno,english]{smfart_app}
\usepackage{aeguill}
\usepackage{enumerate}
\usepackage{amssymb,amsmath,latexsym,amsthm}
\usepackage[T1]{fontenc}
\usepackage{geometry}
\usepackage{url}
\usepackage[frenchb, main=english]{babel}
\usepackage[utf8]{inputenc}
\usepackage{mathrsfs}
\usepackage{xcolor}
\usepackage{comment}
\definecolor{violet}{rgb}{0.0,0.2,0.7}
\definecolor{rouge2}{rgb}{0.8,0.0,0.2}
\usepackage{tikz}
\usepackage{empheq}
\usepackage{tikz-cd}
\usetikzlibrary{matrix,arrows,decorations.pathmorphing}
\usepackage{hyperref}
\usepackage{mathpazo}
\hypersetup{
    bookmarks=true,         
    unicode=false,          
    pdftoolbar=true,        
    pdfmenubar=true,        
    pdffitwindow=false,     
    pdfstartview={FitH},    
    pdftitle={},    
    pdfauthor={},     
    colorlinks=true,       
   linkcolor=rouge2,          
    citecolor=violet,        
    filecolor=black,      
    urlcolor=cyan}           
\usepackage{enumitem}
\usepackage{appendix}

\usepackage{newunicodechar}
\newunicodechar{°}{\ensuremath{^\circ}}
\renewcommand{\P}{\mathbb{P}}
\newcommand{\R}{\mathbb{R}}
\newcommand{\CC}{\mathbb{C}}
\newcommand{\Q}{\mathbb{Q}}
\newcommand{\DD}{\mathbb D}
\newcommand{\Z}{\mathbb{Z}}
\newcommand{\N}{\mathbb{N}}
\newcommand{\B}{\mathbb{B}}

\renewcommand{\d}{\partial}

\newcommand{\ddbar}{\partial\bar{\partial}}

\newcommand{\vp}{\varphi}

\newcommand{\wt}{\widetilde}
\newcommand{\cX}{\mathcal{X}}
\newcommand{\cL}{\mathcal{L}}
\newcommand{\cC}{\mathcal{C}}

\renewcommand{\O}{\mathcal{O}}

\newcommand{\ep}{\varepsilon}
\renewcommand{\epsilon}{\varepsilon}

\newcommand{\VM}{\mathrm{VM} \,}

\newcommand{\la}{\langle}

\newcommand{\ra}{\rangle}
\newcommand{\ol}{\overline}

\renewcommand{\ge}{\geqslant}
\renewcommand{\le}{\leqslant}
\renewcommand{\leq}{\leqslant}
\renewcommand{\geq}{\geqslant}
\newcommand{\Ric}{\mathrm{Ric} \,}
\newcommand{\codim}{\mathrm{codim}}
\newcommand{\Tr}{\mathrm{Tr} \,}
\newcommand{\om}{\omega}

\newcommand{\ddc}{dd^c}

\newcommand{\Yz}{Y°}
\newcommand{\Xz}{X°}

\newcommand{\pse}{\psi_{\varepsilon}}

\newcommand{\omvp}{\omega_{\varphi}}
\newcommand{\omt}{\om_{t}}

\newcommand{\omvpe}{\om_{\vp_{\ep}}}
\newcommand{\re}{\rho_{\ep}}

\newcommand{\Supp}{\mathrm {Supp}}
\newcommand{\tr}{\mathrm{tr}}

\renewcommand{\tt}{\theta_t}

\newcommand{\ssm}{\smallsetminus}

\newcommand{\td}{\tau_{\delta}}
\newcommand{\ud}{u_{\delta}}
\newcommand{\vd}{v_{\delta}}
\newcommand{\vdj}{v_{\delta_j}}

\newcommand{\Cdy}{C_{\delta,y}}

\newcommand{\dbar}{\bar \partial}

\newcommand{\Ommx}{(\Omega \wedge \overline \Omega)^{\frac 1m}}
\newcommand{\Omm}{(\Omega_y \wedge \overline \Omega_y)^{\frac 1m}}

\newcommand{\cF}{\mathcal F}

\newcommand{\pdt}{p^* (id\underline t \wedge \overline{ d \underline t}) }
\newcommand{\Pic}{\mathrm{Pic}}

\newcommand{\Cdt}{C_{\delta}(t)}

\usepackage{graphicx}
\newcommand{\intprod}{\mathbin{\raisebox{\depth}{\scalebox{1}[-1]{$\lnot$}}}}

\numberwithin{equation}{section}

\setdescription{labelindent=\parindent, leftmargin=2\parindent}
\setitemize[1]{labelindent=\parindent, leftmargin=2\parindent}
\setenumerate[1]{labelindent=0cm, leftmargin=*, widest=iiii}

\theoremstyle{plain}
\newtheorem{theo}{Theorem}[section]
\newtheorem{prop}[theo]{Proposition}
\newtheorem{coro}[theo]{Corollary}

\newtheorem{lemm}[theo]{Lemma}

\newtheorem{rema}[theo]{Remark}

\setlist[enumerate]{label=(\thetheo.\arabic*), before={\setcounter{enumi}{\value{equation}}}, after={\setcounter{equation}{\value{enumi}}}}

\newtheorem{bigthm}{Theorem}
\newtheorem{bigcoro}[bigthm]{Corollary}

\newtheorem{app}{Proposition}
\newtheorem{appcoro}[app]{Corollary}
\renewcommand{\theequation}{\thesection  .\!\! \arabic{equation}}

\begin{document}

\title[Variation of singular Kähler-Einstein metrics]{Variation of singular Kähler-Einstein metrics:\\ Kodaira dimension zero}

%

\author{Junyan Cao}
\address{Institut de Mathématiques de Jussieu, Université Paris 6, 4 place Jussieu, 75252 Paris, France}
\email{junyan.cao@imj-prg.fr}

\author{Henri Guenancia}
\address{Institut de Mathématiques de Toulouse; UMR 5219, Université de Toulouse; CNRS, UPS, 118 route de Narbonne, F-31062 Toulouse Cedex 9, France}
\email{henri.guenancia@math.cnrs.fr}

\author{Mihai P\u{a}un}
\address{Institut für Mathematik, Universität Bayreuth, 95440 Bayreuth, Germany}
\email{mihai.paun@uni-bayreuth.de}


\begin{abstract}
  We study several questions involving relative Ricci-flat Kähler metrics for families of log Calabi-Yau manifolds. Our main result states that if  $p:(X,B)\to Y$ is a  Kähler fiber space such that $\displaystyle (X_y, B|_{X_y})$ is generically klt, $K_{X/Y}+B$ is relatively trivial and $p_*(m(K_{X/Y}+B))$ is Hermitian flat for some suitable integer $m$, then $p$ is locally trivial. Motivated by questions in birational geometry,
  we investigate the regularity of the relative singular Ricci-flat Kähler metric corresponding to a family $p:(X,B)\to Y$ of klt pairs $(X_y,B_y)$ such that $\kappa(K_{X_y}+B_y)=0$. Finally, we disprove a
  folkore conjecture by exhibiting a one-dimensional family of elliptic curves whose relative (Ricci-) flat metric is not semipositive.
\end{abstract}

\date{\today}
\keywords{Kähler fiber space, log Calabi-Yau manifolds, conic Kähler metrics, direct image of log pluricanonical bundles }
\subjclass{14J10, 14J32, 32Q20}
\maketitle
\tableofcontents

\section*{Introduction}

\noindent In this article we continue our study of fiber-wise singular K\"ahler-Einstein
metrics started in \cite{JHM1} in the following context.

Let $p:(X, B)\to Y$ be a Kähler fiber space, where $B$ is an effective
divisor such that $\displaystyle (X_y, B|_{X_y})$ is klt for all $y\in Y$ in the complement of some analytic subset of the base $Y$. We are interested here in the curvature and regularity properties of the metric induced on $K_{X/Y}+ B$ by the canonical metrics on fibers $X_y$ under the hypothesis
$$\kappa(K_{X_y}+B_y)=0.$$
The far reaching goal we are pursuing here is a criteria for the birational
equivalence of the fibers $\displaystyle (X_y, B|_{X_y})$ of $p$
in a geometric context inspired by results due to E. Viehweg, Y. Kawamata and J.~Koll\'ar in connection with the $C_{nm}$ conjecture. To this end,
the fiber-wise K\"ahler-Einstein metrics are playing a crucial role. 
Due to some technical difficulties --which we hope to overcome in a forthcoming paper-- our most complete results are obtained under the more restrictive
hypothesis $c_1(K_{X_y}+B_y)=0$, i.e. in the absence of base points of log-canonical bundle of fibers.

\subsection*{Main results}

Let $p:(X, B)\to Y$ be a proper, holomorphic fibration between two Kähler manifolds, where $B=\sum b_iB_i$ is an effective $\Q$-divisor on $X$ whose coefficients $b_i\in (0,1)$ are smaller then one.
We assume that there exists $Y°\subset Y$  contained in the smooth locus of $p$ such that $B|_{X_y}$ has snc support and set $X° :=p^{-1} (Y°)$.
The fibers of $p$ are assumed to satisfy
$$c_1(K_{X_y}+B|_{X_y})=0 \quad \mbox{ for any } y \in Y°.$$
If we fix a reference Kähler form $\om$ on $X$, then we can construct a fiberwise Ricci-flat conic Kähler $\theta_y$ metric, i.e. a solution of the equation
$$
\begin{cases}
\Ric \theta_y= [B_y] \\
 \theta_y\in [\om_y]
 \end{cases}.$$
There exists a unique function $\vp\in L^1_{\rm loc}(X^{\circ})$ such that
$$
\begin{cases}
\theta_y=\om_y+\ddc \vp|_{X_y}  \\
\int_{X_y} \vp \,\om_y^n=0
 \end{cases}.$$
 The closed $(1,1)$-current $\theta_{\rm KE}°:=\om+\ddc \vp$ on $X^{\circ}$ is called \emph{relative Ricci-flat conic Kähler metric} in $[\om]$.
As we shall soon see, the current
$\theta_{\rm KE}°$ is not positive in general, which marks an important
difference with the case of K\"ahler fiber spaces whose generic fiber is of
(log) general type.
\medskip

\noindent Nevertheless, we establish here the following result (cf. Theorem~\ref{foliationiso} for a complete version).

\begin{bigthm}
\label{thmb}
Let $p:(X, B)\to Y$ be a map as above, and let $\om$ be a fixed Kähler metric on $X$. Assume that the following conditions are satisfied.
\begin{enumerate}
\item[$(i)$] For $y\in Y°$, the $\Q$-line bundle $K_{X_y}+B_y$ is numerically trivial.
\item[$(ii)$] For some $m$ large enough, the line bundle $p_*(m(K_{X°/Y°}+B))$ is Hermitian flat with respect to the Narasimhan-Simha metric $h$ on $Y°$, cf. \eqref{metrcon}.
\end{enumerate}

\noindent
Then we can construct a $(1,1)$-current $\theta°_{\rm KE}$ such that the restriction $\theta_y$ of $\theta°_{\rm KE}$ to $X_y$ is a representative of $\{\om\}|_{X_y}$ and solves $\Ric \theta_y=[B_y]$. Moreover we have
\begin{enumerate}
\item[$(\dagger)$] $\theta°_{\rm KE}$ is positive and it extends canonically to a closed positive current $\theta_{\rm KE}\in\{\om\}$ on $X$.
\item[$(\ddagger)$]  The fibration $(X,B)\to Y$ is locally trivial over $Y°$. Moreover, if $p$ is smooth in codimension one and $\codim_X(B\ssm X°)>1$, then $p$ is locally trivial over the whole $Y$.
\end{enumerate}
\end{bigthm}

The result above has many geometric applications, like for instance a K\"ahler version of a theorem of Ambro \cite{Amb05}, cf Corollary~\ref{ambro} and its proof given in page~\pageref{pageambro}.

\noindent
Another striking consequence is the following positivity property of direct images of pluri-log canonical bundles cf page~\pageref{bigdirp} for a proof. It can be seen as a logarithmic version of Viehweg's $Q_{n,m}$-conjecture for families of log Calabi-Yau manifolds, cf \cite{Vieh83b}.

\begin{bigcoro}
\label{bigdir}
 Let $p:(X, B)\to Y$ be a fibration between two compact K\"ahler manifolds such that  $\displaystyle c_1\left(K_{X_y}+ B|_{X_y}\right)= 0$ for a generic $y\in Y$. Assume moreover that the logarithmic Kodaira-Spencer map
  \begin{equation}\label{ks}
    T_Y\to {\mathcal R}^1p_\star\left(T_{X/Y}(-\log  B)\right)
  \end{equation}
is generically injective. Then the bundle $p_\star\left(m(K_{X/Y}+ B)\right)^{\star\star}$ is big.
\end{bigcoro}

\noindent We remark that, based on Corollary \ref{bigdir} and some deep tools, Y. Deng \cite{Deng19}
proved recently the hyperbolicity of bases of maximally variational smooth families of log Calabi-Yau pairs.

%

\bigskip

We are next interested in the following setting
$$\kappa(K_{X_y}+B|_{X_y})=0$$
which is more natural from the birational geometry point of view.
The main result we establish in this context 
is a regularity theorem for the relative Kähler-Einstein metric.
The point is that here we have no further assumptions on the basepoints of $K_{X_y}+B_y$ or the flatness of the direct image of some power of $K_{X/Y}+B$, cf. page~\pageref{thme}.

\begin{bigthm}
\label{thmc}
In the above framework, let $\om$ be a fixed Kähler metric on $X$ and assume that for $y$ generic the Kodaira dimension of $K_{X_y}+B_y$ equals zero. Let $E$ be an effective $\mathbb Q$-divisor
such that $K_{X_y}+B_y \sim_{\Q} E_y$. Then there exists a current $\theta°_{\rm KE}$ of (1,1)-type whose restriction $\displaystyle \theta_y:= \theta°_{\rm KE}|_{X_y}$ is a representative of $\{\om\}|_{X_y}$ and solves the equation $\Ric \theta_y=-[E_y]+[B_y]$.

\noindent
In addition, the local potentials of $\theta°_{\rm KE}$ are Lipschitz on $X°\ssm \Supp(B+E)$.
\end{bigthm}

\bigskip

One may wonder of the assumptions concerning the flatness of the
direct image of the bundle $m(K_{X/Y}+ B)$ cannot be removed in Theorem \ref{thmb}.
Indeed, a folklore conjecture asserts that the form $\theta_{\rm KE}°$ is \textit{semipositive} provided that say $B=0$ and $c_1(X_y)=0.$
By using the results in Appendix, we show that this is simply wrong.

\begin{bigthm}
\label{thma}
There exist a smooth, proper fibration $p:X\to Y$ between Kähler manifolds such that $c_1(X_y)=0$ for all $y\in Y$ and a Kähler form $\om$ on $X$ such that the relative Ricci-flat metric $\theta_{\rm KE}\in [\om]$ is not semipositive.
\end{bigthm}

The example we exhibit is constructed from a special K3 surface admitting a non-isotrivial elliptic fibration as well as another transverse elliptic fibration. The construction is detailed in Section~\ref{nonpositive}. \\

\subsection*{Previously known results}
In connection with Theorem \ref{thmb}, the statements obtained so far
are based on two different type of techniques 
arising from algebraic geometry and complex differential geometry, respectively. One can profitably consult the articles \cite{Vieh83b}, \cite{Kol87} and \cite{Kawa85} for results aimed at the Iitaka conjecture. From the complex differential geometry side we refer to
\cite{Bern11}, \cite{HT15}, \cite{BMW17} and the references therein.

The folklore conjecture that we disprove in Theorem \ref{thma} arose from a result of Schumacher \cite{Schum} who proved the semipositivity of the relative K\"ahler-Einstein metric for families of canonically polarized manifolds (see also the related works \cite{RBerm13}, \cite{Tsuji10}). He also implicitly conjectured that an analogous semipositivity result should hold for families of Calabi-Yau manifolds \cite[p.7]{Schum}, and this was explored in the thesis of Braun \cite{Braun} and in the papers \cite{BCS, BCS2} where positive partial results were obtained. The semipositivity question of $\theta_{\rm KE}$ also appeared in the work \cite{EGZ16} on the K\"ahler-Ricci flow. 

\subsection*{Main steps of the proof} We will describe next the outline of the proof of Theorems~\ref{thmb}, \ref{thmc} and \ref{thma} above. \\

\noindent $\bullet$ The first item of Theorem \ref{thmb} is established by using two ingredients. The first one consists in showing that the conic Ricci-flat metric in $\displaystyle \{\omega_{X_y}\}$ on each fiber $X_y$
is the normalized limit of the unique solution of the family of equations of type
\begin{equation}\label{steps1}
\Ric \rho_\ep = -\rho_\ep +\varepsilon \om + [B]
\end{equation}
on $X_y$ where $\rho_\ep\in \ep\{\omega_{X_y}\}$. We show that $\om°_{\rm KE}|_{X_y}$ is obtained as limit
of $\frac 1{\ep}\rho_\ep$ as $\ep\to 0$. On the other hand, the main result of \cite{Gue16} shows that the family
$\rho_\ep$ has psh variation for each positive $\ep> 0$ and the result follows (the flatness of the direct image is crucial in order to be able to use \cite{Gue16}).

The arguments for the second item of Theorem \ref{thmb} is more involved. We use a different type of approximation of the conic Ricci-flat metric, by regularizing the volume element. Let $\tau_\delta$ be the
resulting family of metrics. The heart of the matter is to show that the horizontal lift with respect to $\tau_\delta$ of any local holomorphic vector field on the base has a holomorphic limit as $\delta\to 0$.
This is a consequence of the estimates in \cite{GP} combined with the PDE satisfied by the geodesic curvature of $\tau_\delta$, cf. \cite{Schum}. Then we show that the geodesic curvature tends to a (positive) constant
and as a consequence we finally infer that the horizontal lift of holomorphic vector fields with
respect to $\om°_{\rm KE}$ is holomorphic and tangent to $B$. \\

\noindent $\bullet$ The equation $\Ric \om=-[E]+[B]$ translates into an Monge-Ampère equation where the right-hand side has poles and zeros. Poles are relatively manageable in the sense that they induce conic metrics, that is we know relatively precisely the behavior of the complex Hessian of the solution. Zeros, however, are much more complicated to deal with for several reasons. First, it seems hard to produce a global degenerate model metric that should encode the behavior of the solution. Next, the regularized solutions of the Kähler-Einstein equation do not satisfy a Ricci lower bound, hence it seems difficult to estimate their Sobolev constant.

In Proposition~\ref{WS}, we establish a uniform (weak) Sobolev inequality where the measure in the right-hand side picks up zeros. Then, we get onto studying the regularity of families of such metrics. Despite having a rather poor understanding of the fiberwise metrics, we are still able to analyze the first order derivatives of the potentials in the transverse directions, leading to an $L^2$ estimate, yet with respect to a more degenerate volume form, cf Proposition~\ref{L2N}. This is however enough to deduce the Lipschitz variation of the potentials away from $\Supp(B+E)$. \\

\noindent $\bullet$ The counterexample provided by Theorem~\ref{thma} is built from an elliptic fibration $p:X\to \mathbb P^1$ where $X$ is a K3 surface. In the Appendix, it is showed that one can find such a fibration with the following properties: its singular fibers are irreducible and reduced, it is not isotrivial and it admits another transverse elliptic fibration. These properties allow to find a semiample, $p$-ample line bundle $L\to X$ with numerical dimension one. Then, the relative Ricci-flat metric $\theta\in c_1(L)|_{X°}$ cannot be semipositive, for otherwise one can show that it would extend to a positive current $\theta\in c_1(L)$ and as $L$ is not big, results of Boucksom show that $$\theta^2\equiv 0 \quad \mbox{on } X°.$$ Using horizontal lifts of $\theta$, one can finally conclude that the foliation $\mathrm{ker}(\theta)$ is holomorphic, induced by a local trivialization of the family This contradicts the non-isotriviality of $p$.
Passing from the relative Ricci-flat metric in $c_1(L)$ to one in a Kähler class can be done using a limiting process. \\

\subsection*{Organization of the paper}
\begin{enumerate}
\item[$\bullet$] \S \ref{sec:rrf} We prove Theorem~\ref{foliationiso}, and then derive successively Corollary~\ref{ambro} and Corollary~\ref{bigdir}.
\item[$\bullet$] \S \ref{sec:transverse}: We obtain transverse regularity results for families of Monge-Ampère equations corresponding to adjoint linear systems having basepoints. This leads to Theorem~\ref{thmc}.
\item[$\bullet$] \S \ref{nonpositive}: We prove Theorem~\ref{counterexample} using results from the Appendix.

\end{enumerate}

\subsection*{Acknowledgments} We would like to thank Sébastien Boucksom, Tristan Collins, Vincent Guedj, Christian Schnell, Song Sun, Valentino Tosatti and Botong Wang for numerous useful discussions about the topics of this paper.
This work has been initiated while H.G. was visiting KIAS, and it was carried on during multiple visits to UIC as well as to IMJ-PRG; he is grateful for the excellent working conditions provided by these institutions. During the preparation of this project, the authors had the opportunity to visit FRIAS on several occasions and benefited from an excellent work environment.

H.G. is partially supported by NSF Grant DMS-1510214, and M.P is partially supported by NSF Grant DMS-1707661 and Marie S. Curie FCFP.
\setcounter{tocdepth}{2}

\section{Relative Ricci-flat conic metrics}
\label{sec:rrf}

\subsection{Setting}
\label{15aout1}
Let $p:X\to Y$ a holomorphic proper map of relative dimension $n$ between K\"ahler manifolds. We denote by $Y°\subset Y$ the set of regular values of $p$, and let $X°:=p^{-1}(Y°)$ so that $p_{|X°}:X°\to Y°$ is a smooth fibration. For $y\in Y°$, one writes $X_y:=p^{-1}(X_y)$ the fiber over $y$.  Let $B$ be an effective $\Q$-divisor on $X$ that has coefficients in $(0,1)$ and whose support has snc.
Our assumption throughout the current section will be that
for each $y\in Y°$ we have
\begin{equation}
\label{mp1}
  c_1(K_{X_y} +B_y)=0 \in H^{1,1} (X_y , \Q).
\end{equation}
Thanks to the log abundance in the Kähler setting, cf. Corollary~\ref{log-ab} on page~\pageref{log-ab}, we know that $K_{X_y} +B_y$ is $\Q$-effective.
Combining this with Ohsawa-Takegoshi extension theorem in its Kähler version, cf. \cite{CaoOT}, one can assume that there exists $m\ge1$ such that $m(K_{X_y}+B_y)\simeq \mathcal O_{X_y}$ for all $y\in Y°$.
\medskip

\noindent In this context the main result we obtain here shows that the flatness of the direct image
$p_\star (m K_{X/Y} +m B)$ implies the local isotriviality of the family $p: (X, B)\to Y$. By this we mean that there exists a holomorphic vector field $v$ on $X°$ whose flow identifies the pairs $\displaystyle (X_y, B_y)$ and $\displaystyle (X_w, B_w)$ provided that $y, w\in Y°$ are close enough. This is the content of Theorem \ref{foliationiso} below. Prior to stating our theorems in a formal manner, we need to recall a few notions and facts.

Given a point $y\in Y°$, there exists a coordinate ball $U\subset Y°$ containing $y$ and a nowhere vanishing
holomorphic section
\begin{equation}\label{mp2}
  \Omega\in H^0\left(X_U, m(K_{X/Y}+B)|_{X_U}\right)
\end{equation}
by our assumption \eqref{mp1}, where $X_U:= p^{-1}(U)$.

If $f_B$ is a local multivalued holomorphic function cutting out the $\Q$-divisor $B$, then the form $\displaystyle \frac{\Omm}{|f_B|^2}$ induces a volume element on the
fibers of $p$ over $U$.
We fix a K\"ahler class $\{\om\}\in H^{1,1}(X,\mathbb R)$. Up to renormalizing $\om$, one can assume that the constant function
$$Y°\ni y \mapsto \int_{X_y}\om^n$$
is identically equal to $1$. We also define $\displaystyle V_y:=\int_{X_y}\frac{\Omm}{|f_B|^2}$; this is a H\"older continuous function of $y\in Y°$.

Let $\rho_y$ be the unique positive current on $X_y$ which is
co\-homologuous to $\omega_y$
and satisfies
$$\rho_y^n = \frac{\Omm}{V_y |f_B|^2 },$$
cf. \cite{Yau}.
One can write $\displaystyle \rho_y=\om|_{X_y}+\ddc \varphi_y$,
where the function $\varphi_y$ is uniquely determined
by the normalization
\begin{equation}
\label{normvp}
\int_{X_y}\varphi_y \,\frac{\Omm}{|f_B|^2 } =0.
\end{equation}
For each $y\in U\subset Y°$, the current $\rho_y$ is reasonably
well understood: it has H\"older potentials, and it is quasi-isometric to a metric with conic singularities along $B$, cf.\ \cite{GP}.

\noindent
We analyze next its regularity properties in the ``base directions''; this will allow us to derive a few interesting geometric consequences.

The function $\vp$ defined on $X°$ by $\vp(x):=\vp_{p(x)}(x)$ is a locally bounded function on $X°$ (by the family version of Ko\l odziej's estimates cf. \cite{DDGHKZ}) hence it induces a $(1,1)$ current
\begin{equation}
\label{rho}
\rho:=\om+\ddc \varphi
\end{equation}
on $X°$.
Let $\Delta\subset Y°$ be a small, 1-dimensional disk. If $\Delta$ is generic enough, then the inverse image $\cX:= p^{-1}(\Delta)$ is non-singular, and the restriction map $p:\cX\to \Delta$ is a submersion.
We denote by
$t$ a holomorphic coordinate on the disk $\Delta$. Following \cite{Siu86}
we recall next the expression of the \emph{horizontal lift} of the local
vector field $\displaystyle \frac{\partial}{\partial t}$.
For the moment, this is a vector field $v_\rho$ with
distribution coefficients on the total space $\cX$ given by the expression
\begin{equation}\label{hlift}
v_\rho:= \frac{\partial}{\partial t}- \sum_\alpha\rho^{\overline{\beta}\alpha}\rho_{t\overline{\beta}}\frac{\partial}{\partial z_{\alpha}},
\end{equation}
where the notations are as follows. We denote by $(z_1, \dots, z_n, t)$
a co-ordinate system centered at some point of $\cX$, and
$\rho_{t{\overline\beta}}$ is the coefficient of $dt\wedge d{\overline{z}_\beta}$. We denote by $\left(\rho^{\overline{\beta}\alpha}\right)$ the coefficients of the inverse of the matrix $\left(\rho_{\alpha \overline{\beta}}\right)$.

\medskip

\noindent The reflexive hull of the direct image
\begin{equation}
\label{mp4}
\cF_m:= p_\star\left(m(K_{X/Y}+ B)\right)^{\star\star}
\end{equation}
plays a key role in study of the geometry of algebraic fiber spaces.
It admits a positively curved singular metric whose construction
we next recall, cf. \cite{BP, PT} and the references therein.

Let $\sigma\in H^0(U, \cF_m|_U)$ be a local holomorphic section of the
line bundle $\cF_m$ defined over a small coordinate set $U\subset Y°$.
The expression
\begin{equation}
  \label{metrcon}
\Vert \sigma\Vert^2_y:= V_y^{m-1}\int_{X_y}\frac{|\sigma|^2}{|\Omega_y|^{2\frac{m-1}{m}}}e^{-\phi_B}
\end{equation}
defines a metric $h$ on $\displaystyle \cF_m|_{Y°}$. It is remarkable that this metric extends across the singularities of the map $p$, and it has semi-positive curvature current, see \emph{loc. cit.} for more complete statements.

\medskip

\subsection{Main results}
This sub-section aims to the proof of the following results.

\begin{theo}
\label{thm:cyvar}
Let $p:(X, B)\to Y$ be a proper holomorphic map between K\"ahler manifolds as in \eqref{15aout1}. We assume moreover that the curvature of $\cF_m$ with respect to the metric in \eqref{metrcon} equals zero when restricted to $Y°$.
Then the $(1,1)$-current $\rho$ defined on $X°$ by \eqref{rho} is semipositive and it extends canonically to a closed positive current on $X$ in the cohomology class $\{\om\}$.
\end{theo}

 For example, if we assume that $Y$ is compact, then the curvature of $\cF_m$ will automatically be
zero if $c_1(\cF_m)= 0$ thanks to the properties of the metric
\eqref{metrcon} discussed above, cf \cite[Thm.~5.2]{CP17}.

What we mean by the word canonical in the Theorem~\ref{thm:cyvar} above is that the local potential $\vp$ of $\rho$ are locally bounded above across $X\ssm X°$.

\medskip

\noindent We equally prove the next statement.

\begin{theo}
\label{foliationiso}
We assume that the hypothesis in Theorem \ref{thm:cyvar} are satisfied. Then, $p$ is locally trivial over $Y°$, that is, for every $y\in Y°$, there exists a neighborhood $ U \subset Y°$ of $y$ such that
$$(p^{-1}(U), B) \simeq (X_y, B|_{X_y}) \times U .$$
Moreover, if $p$ is smooth in codimension one, then $p$ is locally trivial over the whole $Y$ provided that $\codim_{X\ssm X°}(B\ssm X°)>0$.

\end{theo}

\noindent In particular, under the assumptions in the second part of Theorem \ref{foliationiso} the map $p$ is automatically a locally isotrivial submersion.

\medskip

\noindent As an application, we establish the following result; it partially generalizes to the K\"ahler case a theorem of F. Ambro \cite{Amb05}.

\begin{coro}
\label{ambro}
Let $p: X \rightarrow Y$ be a fibration between two compact K\"ahler manifolds. Let $B$ be a $\Q$-effective klt divisor on $X$ with snc support.
\begin{enumerate}

\item \label{part1cor} If $-(K_X +B)$ is nef, then $-K_Y$ is pseudo-effective.

\item \label{part2cor} Moreover, if $c_1 (K_X +B)=0$ and $c_1 (Y)=0$, then $p$ is locally trivial, that is, for every $y\in Y$, there exists a neighborhood $U \subset Y$ of $y$ such that
$$(p^{-1}(U), B) \simeq (X_y, B|_{X_y}) \times U .$$
In particular, if $c_1 (K_X +B)=0$, the Albanese map $p: X\rightarrow Alb (X)$ is locally trivial.
\end{enumerate}
\end{coro}

\subsection{Proof of Theorem \ref{thm:cyvar}}
We will proceed by approximation, mainly using the following lemma
combined with the results in \cite{Gue16}.
\medskip

\noindent The next statement will enable us to reduce the problem to the
canonically polarized pairs.
\begin{lemm}
\label{prop:conv}
Let $X$ be a compact K\"ahler manifold and let $B$ be an effective divisor such that $(X, B)$ is klt. We assume that $c_1 (K_X +B)=0$.
Let $\om$ be K\"ahler form on $X$.
For every $\ep>0$,  let $\rho_\ep \in \ep\{\om\}$ be the unique twisted conic K\"ahler-Einstein metric such that
\begin{equation}
\label{mp}
\Ric \rho_\ep=-\rho_\ep+\ep \om +[B] .
\end{equation}
Let $\rho \in \{\omega\}$ be the unique conic K\"ahler-Einstein metric such that $\Ric \rho = [B]$. Then
$$\lim_{\ep \to 0}\frac{1}{\ep} \rho_\ep=\rho$$
where the convergence is smooth outside $\Supp(B)$.
\end{lemm}

\begin{proof}
Let $m\in\N$ such  that $m (K_X +B)$ is effective. Let $\Omega\in H^0 (X, m (K_X +B))$ be a holomorphic section normalized such that
\begin{equation}
\label{mp10}
\int_X \frac{\Ommx}{|f_B|^2 } =1.
\end{equation}
There exists a unique function $\vp_{\ep}$ on $X$ such that
\begin{enumerate}
\item $\rho_{\ep}=\ep \om+\ddc \vp_{\ep}$
\item $\rho_{\ep}^n=\ep^n e^{\vp_{\ep}} \frac{\Ommx}{|f_B|^2 }$
\end{enumerate}

\noindent
Now, let us set $\psi_{\ep}:=\frac 1{\ep}\vp_{\ep}$. One has $\frac 1{\ep} \rho_{\ep}=\om+\ddc \psi_{\ep}$ and
\begin{equation}
\label{MAe} (\om+\ddc \psi_{\ep})^n=e^{\ep \psi_{\ep}} \frac{\Ommx}{|f_B|^2 }
\end{equation}
As $\frac{\Ommx}{|f_B|^2 }$ and $(\frac 1{\ep} \rho_{\ep})^n$ are probability measures and $\psi_{\ep}$ is $\om$-psh,
Jensen inequality yields $\int_{X}(\ep \psi_{\ep}) \,\frac{\Ommx}{ |f_B|^2 } \le 0$, and therefore
\begin{equation}
\label{intne}
\int_X \psi_{\ep} \,\frac{\Ommx}{ |f_B|^2 } \le 0.
\end{equation}
As the measure $\frac{\Ommx}{|f_B|^2 }$ integrates every quasi-psh function,
it follows from standard results in pluripotential theory that there exists a constant $C$ such that
\begin{equation}
\label{sup}
\sup_{X}\psi_{\ep}\le C
\end{equation}
By \eqref{MAe}-\eqref{sup} and Ko\l odziej's estimate, one gets
\begin{equation}
\label{osc}
\mathrm{osc}_{X}\psi_{\ep}\le C \\
\end{equation}

\noindent
As $\frac{\Ommx}{ |f_B|^2 }$ and $(\frac 1{\ep} \rho_{\ep})^n$ are probability measures again,
\eqref{MAe} shows that
$$\inf_{X}\psi_{\ep}\le 0 \le \sup_{X}\psi_{\ep} .$$ Combining this information with \eqref{osc}, we obtain the inequality
\begin{equation}
\label{sup3}
|| \psi_{\ep}||_{L^{\infty}(X)}\le C.
\end{equation}
Moreover, Jensen inequality applied to the equation $ \frac{\Ommx}{|f_B|^2 }= e^{-\ep \psi_{\ep}}(\frac 1{\ep} \rho_{\ep})^n$ yields
\begin{equation}
\label{intpo}
\int_{X}\psi_{\ep} \,(\om+\ddc \psi_{\ep})^n \ge 0
\end{equation}
From \eqref{MAe} and \eqref{sup3}, we get uniform estimates at any order for $ \psi_{\ep}$ outside $B$.
If $\psi$ is a subsequential limit of the family $( \psi_{\ep})_{\ep>0}$, it will satisfy
$$(\om+\ddc \psi)^n= \frac{\Ommx}{|f_B|^2 }$$
Combining this information with \eqref{intne} and \eqref{intpo}, we find
$$\int_{X}  \psi \, \frac{\Ommx}{ |f_B|^2 }=0$$
Therefore $\psi$ is uniquely determined, and the whole family $( \psi_{\ep})_{\ep>0}$ converges to $\psi$.
The lemma is thus proved.
\end{proof}
\smallskip

\noindent Now we can prove  Theorem~\ref{thm:cyvar}.

\begin{proof}[Proof of Theorem \ref{thm:cyvar}]
  We fix a reference K\"ahler form $\om$ on $X$, and let $U$ be some small topological open set of $Y°$.
  By hypothesis, the curvature of the bundle $\cF_m|_U$ is identically zero. By using parallel transport, this is equivalent to the existence of a section
\begin{equation}\label{mp11}
    s \in H^0 \left(X_U, m K_{X/Y}+mB|_{X_U}\right)
\end{equation}
whose norm is a constant function on $U$, namely $\| s\|_h (y) =1$ for every $y\in U$.
Let
$$\Omega_y := s |_{X_y} \in H^0 (X_y, m K_{X_y} +m B_y)$$
be the restriction of $s$ to the fibers of $p$.

\noindent
Since $c_1(K_{X_y}+B_y) + \ep\om_{|X_y}$ is a K\"ahler class
for each $\ep > 0$ and for each $y\in Y°$, there exists a unique $\vp_{\ep}$ such that
$$(\ep \om+\ddc \vp_{\ep})^n = \ep^n e^{\vp_{\ep}}   \frac{\Omm}{ |f_B|^2 } \qquad\text{on }X_y. $$
Since $y\in U$ is a regular value, this is equivalent to
$$\Ric{\rho_{\ep,y}}=-\rho_{\ep,y} +\ep\om +[B_y] \qquad\text{on }X_y ,$$
where $\rho_{\ep,y}= \ep \om+\ddc \vp_{\ep} |_{X_y}$.

Next, the section $s$ is holomorphic hence the relative $B$-valued volume forms $\Omm$ induce a metric with zero curvature on $K_{X/Y}+B$ over $p^{-1}(U)$. Because of that,
 $$\rho_\ep :=\ep \om+\ddc \vp_{\ep} $$
  coincides with the current studied in \cite{Gue16} and the content of the main theorem in \textit{loc. cit.} is that $\re$ is positive on $p^{-1}(U)$. Thanks to Lemma \ref{prop:conv}, $\rho$ is the fiberwise weak limit on $p^{-1}(U)$ of the fiberwise twisted K\"ahler-Einstein metrics $\frac1{\ep}\rho_{\ep}$; moreover, the estimate \eqref{sup3} is uniform over $U$, so that $\rho$ is actually the global weak limit of the metrics $\frac{1}{\ep} \re$ on $p^{-1}(U)$. In particular, $\rho \ge 0$ on $p^{-1}(U)$, hence on $X°$.

As for the extension property, it is proved in \cite{Gue16} that $\re$ extends canonically to the whole $X$ as a positive current in $\{\ep \om\}$. This means that given any small neigborhood $U$ of a point $x\in X\ssm X°$, one has $\sup_{U\cap X°} \pse <+\infty$. In other words, $\pse$ extends to an $\om$-psh function on $X$. Now, let us fix $U$ as above. The family of $\om$-psh functions $(\widetilde \psi_{\ep})_{\ep >0}$ on $U$ defined by $\widetilde \psi_{\ep}:=\psi_{\ep}- \sup_U \psi_{\ep}$ is relatively compact. In particular one can find a sequence $\ep_k \to 0$ and an $\om$-psh function $\widetilde \psi$ on $U$ such that $\widetilde \psi_{\ep_k} \to \wt \psi $ a.e. in $U$. Moreover, we know that $\psi_{\ep_k}=\wt \psi_{\ep_k}+\sup_U \psi_{\ep_k}$ converges to the $\om$-psh function $\vp$ a.e. in $U\cap X°$. This implies that $\sup_{U}\psi_{\ep_k}$ converges when $k\to +\infty$. By Hartogs lemma, this implies that $\sup_{U\cap X°} \vp<+\infty$, which had to be proved.
 \end{proof}

\subsection{Proof of Theorem \ref{foliationiso}}
We will proceed in a few steps, roughly as follows.
\smallskip

\noindent $\bullet$ We start by approximating $\rho$ by smoothing the
volume element. Let $\tau_\delta$ be the resulting $\cC^\infty$ form. Then we have $\lim_\delta\tau_\delta= \rho$ in weak sense.
\smallskip

\noindent $\bullet$ We analyze next the behavior of the geodesic curvature of $\tau_\delta$. The main tools are the Laplace equation satisfied by this quantity, cf. \cite{Schum}, and the $\cC^2$-estimates for conic Monge-Amp\`ere equations, cf. \cite{GP}. As a consequence, we first show that
we can extract a limit of
the horizontal lift $v_\delta$ (corresponding to $\tau_\delta$) which is \emph{holomorphic} on the fibers of $p$.
Afterwards we show that the geodesic curvature of
$\tau_\delta$ converges (on $\cX\setminus \Supp(B)$) to a constant as $\delta\to 0$. Finally,
we infer that $v_\delta$ converges to $v_\rho$ uniformly on the complement of the divisor $B$.
\smallskip

\noindent $\bullet$ After completing the previous steps,
we show that $v_\rho$ is in fact holomorphic on the total space $\cX$ by using
a few arguments borrowed from \cite{BoBJDG}.
\smallskip

\noindent $\bullet$ Finally, we show that $v_{\rho}$ extends across the singular locus of $p$ provided that $X$ is compact and $p$ is smooth in codimension one.
\medskip


\subsubsection{Approximation} This is a fairly standard
and widely used procedure, so we will be very brief.

\noindent By hypothesis, we have $B= \sum a_jB_j$ where $a_j\in (0, 1)$ and
$\cup B_j$ has simple normal crossings.
We consider a smooth metric $e^{-\phi_j}$ on the bundle associated to $B_j$; it induces a smooth metric $e^{-\phi_B}:=e^{-\sum a_j \phi_j}$ on the $\Q$-line bundle associated to $B$.
For any $\delta\ge0$ we define the quantity $\Cdy$ as follows
$$e^{-\Cdy}=\int_{X_y}\frac{\Omm}{\prod_j\left(|f_{j}|^2+\delta^2e^{\phi_j}\right)^{a_j} }.$$
Here $\Omega$ is a section of $\cF_m|_\Delta$ whose norm is equal to one at each point, and $f_{j}$ is a local holomorphic function cutting out $B_j$. The expression
$\prod_j \left(| f_{j}|^2+\delta^2e^{\phi_j}\right)^{-a_j}$ is then a globally defined smooth metric on the $\Q$-line bundle associated to $B$. Finally, we let $s_j$ be the canonical section of $\O_X(B_j)$, and we will denote by $|s_j|^2$ the squared norm of $s_j$ with respect to $e^{-\phi_j}$.

\noindent Let us further define the smooth $(1,1)$-form
\begin{equation}\label{mp13}
  \td=\om+\ddc \ud
\end{equation}
on $X°$ such that $\ud |_{X_y}$ is solution of the following system of equations
\begin{equation}
\label{cases}
\begin{cases}
 \,\, (\om+\ddc \ud)^n=e^{\Cdy}\frac{\Omm}{\prod_j\left(|f_{j}|^2+\delta^2e^{\phi_j}\right)^{a_j} }  & \\
 \,\, \int_{X_y}\ud \frac{\Omm}{\prod_j\left(|f_{j}|^2+\delta^2e^{\phi_j}\right)^{a_j}} =0
 \end{cases}.
 \end{equation}
 By the family version of Ko\l odziej's estimates \cite{DDGHKZ}, one can easily see that for any relatively compact subset $U\Subset Y°$, there exists a constant $C>0$ independent of $\delta \in (0,1)$ such that
 \begin{equation}
\label{supsupsup}
\sup_{y\in U}|| \ud||_{L^{\infty}(X_y)}\le C
\end{equation}
As a consequence, we get the following easy result, cf  \eqref{rho} for the definition of $\rho$ and $\varphi$.
\begin{lemm}
\label{conv0}
When $\delta$ approaches zero, $\td$ converges weakly to $\rho$ on $X°$. More precisely, one has $\ud \to \vp$ in $L^1_{\rm loc}(X°)$.
\end{lemm}
\begin{proof}
The convergence $\ud \to \vp$ in $L^1_{\rm loc}(X_y)$ follows from Ko\l odziej's stability theorem \cite[Thm.~4.1]{Kolo03} (one even gets uniform convergence). The convergence on the total space then follows from Lebesgue's dominated convergence theorem coupled with \eqref{supsupsup}.
\end{proof}

\subsubsection{Uniformity properties of $(\tau_\delta)_{\delta> 0}$}
In this subsection we will only consider the restriction of our initial family
of manifolds above a disk in the complex plane
\begin{equation}
  \label{mp20}
  p:\cX\to \Delta
\end{equation}
where we recall that $\Delta\subset Y°$ is generic and $\cX= p^{-1}(\Delta)$.

The coordinate on $\Delta$ will be denoted by $t$. We recall that the
geodesic curvature of the form $\tau_\delta$ is the function defined by the equality
\begin{equation}
  \label{mp21}
  \tau_\delta^{n+1}= c(\tau_\delta)\tau_\delta^n\wedge \sqrt{-1}dt\wedge d{\overline t}
\end{equation}
If $v_\delta$ is the horizontal lift of $\displaystyle \frac{\partial}{\partial t}$ with respect to $\tau_\delta$, then it is easy to verify that we have
\begin{equation}
  \label{mp22}
c(\tau_\delta)= \langle v_\delta, v_\delta\rangle_{\tau_\delta}.
\end{equation}
For each $\delta> 0$, the form $\tau_\delta$ induces a metric say $h_{\delta}$ on the relative canonical bundle $K_{\cX/\Delta}$ as follows.
Let $z_1,\dots,z_n, z_{n+1}$ be a coordinate system defined on the set
$W\subset \cX$. Recall that $t$ is a coordinate on $\Delta$. This data induces in particular a trivialization of $K_{\cX/\Delta}$, with respect to which the weight of $h_\delta$ is given as follows
\begin{equation}
  \label{mp23}
e^{\Psi_\delta(z, t)}dz_1\wedge\dots\wedge dz_{n+1}= \tau_\delta^n\wedge \sqrt{-1}dt\wedge d{\overline t}.
\end{equation}
The curvature of $(K_{\cX/\Delta}, h_\delta)$ is the Hessian of the weight
\begin{equation}
  \label{mp24}
  \Theta_\delta(K_{\cX/\Delta})|_W= dd^c\Psi_\delta.
\end{equation}
\medskip

\noindent We have the following result, relating the various quantities defined above.

\begin{lemm}\label{Laplace} Let $\Delta_\delta^{\prime\prime}$ be the Laplace operator corresponding to the metric $\displaystyle\tau_\delta|_{X_t}$. Then we have
  the equality
\begin{equation}
  \label{mp25}
  -\Delta_\delta^{\prime\prime}\left(c(\tau_\delta)\right)= |\dbar v_\delta|^2- \Theta_\delta(K_{\cX/\Delta})(v_\delta, v_\delta).
  \end{equation}
\end{lemm}
\noindent We will not prove Lemma \ref{Laplace} in detail because this type of results appear in many articles, cf. \cite{Schum} or \cite{Paun12}. The main steps are as follows: we have $\displaystyle \Psi_\delta= \log \det(g_{\alpha \overline\beta})$
where we denote $\displaystyle g_{\alpha \overline\beta}:= \tau_{\delta, \alpha \bar \beta}$
and a few simple computations show that the Hessian of $\Psi_\delta$ evaluated in the $v_\delta$-direction equals

\begin{eqnarray}\label{equa534}
\ddbar\log\det(g_{\alpha\ol\beta})(v_\delta, \ol v_\delta)&=& g^{\alpha\ol \beta}g_{ t\ol t, \alpha \ol \beta}-
g^{\alpha\ol \gamma}g^{\delta\ol \beta}g_{\gamma \ol \delta, \ol t}g_{\alpha\ol \beta, t}
\nonumber\\
&- & g^{\alpha\ol \beta}g_{\alpha \ol \beta, \gamma \ol t}g^{\gamma\ol\mu}g_{t\ol \mu}-
g^{\alpha\ol \beta}g_{\alpha \ol \beta, t \ol \gamma} g^{\mu\ol\gamma}g_{\mu \ol t} \\
& + &
g^{\alpha\ol \beta}g_{\alpha \ol \beta, \gamma\ol \tau}g^{\gamma\ol \mu}g^{\rho\ol \tau}g_{t\ol \mu}g_{\rho \ol t}. \nonumber\\
\nonumber
\end{eqnarray}

\smallskip

\noindent In the rhs term in \eqref{equa534} we recognize the beginning of
$\Delta_\delta^{\prime\prime}\left(c(\tau_\delta)\right)$ (cf. the 1st term), and in the end this gives \eqref{mp25}. Again, we refer to \cite{LongChamp}, pages 18-19
for a detailed account of these considerations.
\qed

\begin{rema}{\rm The equation \eqref{mp25} can be seen as the analogue of the usual ${\mathcal C}^2$ estimates in ``normal directions". By this we mean the following: the ${\mathcal C}^2$ estimates are derived by evaluating the Laplace of the (log of the) sum of eigenvalues of the solution metric with respect to the reference metric. Vaguely speaking, in \eqref{mp25} we compute the Laplace of the normal eigenvalue.}
\end{rema}

\medskip

\noindent The following result is an important step towards the proof of
Theorem \ref{foliationiso}.

\begin{prop}
\label{conv}
Let $t\in \Delta$ be fixed. For any sequence $\delta_j \to 0$,  there exists a holomorphic vector field $w$ on $X_t \ssm \Supp(B)$ such that, up to extracting a subsequence, the sequence $(\vdj|_{X_t})_{j\ge 0}$ converges locally smoothly outside $\Supp(B)$ to the vector field $w$.
\end{prop}
\begin{rema}{\rm
At this point, it is not obvious that $w$ is independent of the sequence $\delta_j$ and that it should coincide with to be the lift $v$ of $\frac{\d}{\d t}$ with respect to $\rho|_{X°\smallsetminus \Supp(B)}$.}
\end{rema}

\noindent Before giving the proof of Proposition \ref{conv} we collect here
a few results concerning the family of forms $\displaystyle (\tau_\delta)_{\delta> 0}$ taken from \cite{GP}
and \cite{Gue16}.

\begin{enumerate}
\smallskip

\item[(a)] It follows from \cite[Sect.~5.2]{GP} that $\td|_{X_y}$ has "uniform regularized conic singularities" in the sense that if
 on a small coordinate open set $\Omega \subset \cX$, the divisor $B$ is given be $B=\sum_1 ^r a_j B_j$ where $B_j$ is defined by $\{z_j =0\}$, then there is a constant $C$ independent of $\delta$ such that for any $y\in U$, we have

 {\footnotesize
\begin{equation}
\label{coniccontr}
C^{-1} \Big(\sum_{k=1} ^r \frac{i d z_k \wedge d \overline{z}_k}{(|z_k|^2 + \delta^2)^{a_k}} + \sum_{k \geq r+1} i d z_k \wedge d \overline{z}_k \Big)
\leq \td|_{X_y\cap \Omega} \leq
C \Big(\sum_{k=1} ^r \frac{i d z_k \wedge d \overline{z}_k}{(|z_k|^2 + \delta^2)^{a_k}} + \sum_{k \geq r+1} i d z_k \wedge d \overline{z}_k \Big)
\end{equation}}
\smallskip

\item[(b)] The estimates \cite[(3.13), Prop.~4.1\&4.2]{Gue16} go through for $\ud$, that is, for any integer $k\ge 0$, there exists $C_k>0$ independent of $\delta \in (0,1)$ such that
 \begin{equation}
\label{dt}
\sup_{t\in \Delta} ||\d_t \ud ||_{\mathcal C^k(\Omega \cap X_t)} \le C_k
\end{equation}
and there exists a constant $C>0$ such that the following global estimate holds:
\begin{equation}
\label{2}
\sup_{t\in \Delta} \int_{X_t} |\vd|^2_{\om}  \,\td^n \le C
\end{equation}
One also gets
\begin{equation}
\label{3}
\lim_{\delta \to 0}\sup_{t\in \Delta} \int_{X_t\cap (\cup \{|s_j|^2<\delta\})} |\vd|^2_{\om}  \,\td^n =0
\end{equation}
\end{enumerate}
\smallskip

\noindent Again, we will not reproduce here the arguments for \eqref{dt}--\eqref{3}, but let us comment e.g. \eqref{2} for the comfort of the reader/referee.
The main observation is that in local coordinates this amounts to obtaining a bound
of $\displaystyle \left|\nabla^\delta(\d_t \ud)\right|^2$ with respect to the volume element $\td^n$ on $X_t$. Here $|\cdot |^2$ is measured with respect to the reference metric $\omega$, and $\nabla^\delta$ is the gradient corresponding to
$\td$. By \eqref{coniccontr} this is smaller than $\displaystyle \left|\nabla^\delta(\d_t \ud)\right|_\delta^2$ up to a uniform constant. This new quantity
is controlled by taking the derivative of the Monge-Amp\`ere equation verified by $\td$ in normal directions and integration by parts. Of course, the real proof is much more involved and we refer to \emph{loc. cit.} for the details.
\medskip

\noindent We see immediately that \eqref{2}-\eqref{3} imply the next statement.
\begin{lemm}
\label{4}
One has the following
$$\lim_{\delta \to 0}\sup_{t\in \Delta} \int_{X_t}\left(\sum_j\frac{\delta^2}{|s_j|^2+\delta^2}\right) |\vd|^2_{\om}  \,\td^n =0.$$
\end{lemm}
\medskip

\noindent The proof of Lemma \ref{4} is very elementary and we skip it. We present next the arguments for Proposition \ref{conv}.

\begin{proof}
Recall that in local coordinates,
 $$\vd=\frac{\d}{\d t} -\sum_{\alpha, \beta} \tau_{\delta}^{\bar \beta \alpha }\tau_{\delta, t \bar \beta}\frac{\d}{\d z_{\alpha}}$$
By \eqref{dt}, the family $(\vd|_{X_t})_{\delta>0}$ is relatively compact in the $\mathcal C^{\infty}_{\rm loc}(X_t\ssm \Supp(B))$ topology. Let $\delta_j$ a sequence converging to zero such that $(\vdj|_{X_t})_{j\ge 0}$ converges locally smoothly outside $\Supp(B)$ to a vector field $w$.

\noindent
Now, the geodesic curvature $c(\td)$ of $\td$ satisfies the following equation
\begin{equation}
\label{geodesic}
-\Delta_{\td}c (\td) = \|\dbar \vd\|^2 -\Theta_{ \delta}(K_{\cX/\Delta})(\vd,\bar v_{\delta})
\end{equation}
by Lemma \ref{Laplace}. In our setting (cf. \eqref{mp13} and the definition of $\td$)
the curvature term in \eqref{geodesic} becomes
\begin{equation}
\label{mp33}
  \frac{\partial^2\Cdt}{\partial t\partial \ol t}-
\sum_ja_j\delta^2\frac{\sqrt{-1}\la\partial s_j, \partial s_j\ra
  (\vd,\bar v_{\delta})}{(|s_j|^2+\delta^2)^2}+ \sum a_j \delta^2\frac{\Theta_j(\vd,\bar v_{\delta})}{|s_j|^2+\delta^2}
\end{equation}
where $\Theta_{j}$ above is the curvature of the hermitian line bundle $({\mathcal O}_X(B_j), e^{-\phi_j})$.

\noindent
Integrating \eqref{geodesic} against $\td^n$ yields
\begin{equation}
\label{dbar}
\lim_{\delta \to 0}\sup_{t\in \Delta} \left(\int_{X_t}\|\dbar \vd\|^2 \, \td^n+
\sum_j a_j\int_{X_t}\delta^2\frac{\sqrt{-1}\la\partial s_j, \partial s_j\ra
  (\vd,\bar v_{\delta})}{(|s_j|^2+\delta^2)^2}\, \td^n\right)
= \lim_{\delta \to 0}\sup_{t\in \Delta}\frac{\partial^2\Cdt}{\partial t\partial \ol t}.
\end{equation}
Indeed, thanks to Lemma~\ref{4} the third term in \eqref{mp33} vanishes as $\delta\to 0$.

\noindent
We show next that we have
\begin{equation}
  \label{mp40}
\lim_{\delta \to 0}\sup_{t\in \Delta}\frac{\partial^2\Cdt}{\partial t\partial \ol t}= 0
\end{equation}
and this will end the proof of Proposition~\ref{conv}. Recall that the expression of the function in
\eqref{mp40} is
\begin{equation}\label{mp41}
\Cdt= -\log\int_{X_t}\frac{\Omm}{\prod_j\left(|f_{j}|^2+\delta^2e^{\phi_j}\right)^{a_j} }
\end{equation}
and given that the norm of $\Omega$ is equal to one at each point of $\Delta$,
we have
\begin{equation}\label{mp42}
  \Cdt= -\log\Big(1-
  \int_{X_t}\frac{\prod_j\left(|f_{j}|^2+\delta^2e^{\phi_j}\right)^{a_j}- \prod_j|f_{j}|^{2a_j}}
  {\prod_j|f_{j}|^{2a_j}\prod_j\left(|f_{j}|^2+\delta^2e^{\phi_j}\right)^{a_j}}\Omm\Big).
\end{equation}
With the same notations as in \eqref{coniccontr}, the restriction of the function under the sum sign in \eqref{mp42} on a coordinate set $W_\alpha$ reads as
\begin{equation}\label{mp43}
F_{\alpha, \delta}(z,t):= \frac{\prod_j\left(|z_{j}|^2+\delta^2e^{\phi_j}\right)^{a_j}- \prod_j|z_{j}|^{2a_j}}
  {\prod_j|z_{j}|^{2a_j}\prod_j\left(|z_{j}|^2+\delta^2e^{\phi_j}\right)^{a_j}}
\end{equation}
and then the integral in \eqref{mp42} becomes
\begin{equation}\label{mp44bis}
\sum_\alpha\int_{W_\alpha\cap X_t}\theta_\alpha F_{\alpha, \delta}(z,t)e^{f_\alpha}\omega^n
\end{equation}
where $\theta_\alpha$ is a partition of unit and the $f_{\alpha}$ are given smooth functions. If $v$ is the horizontal lift of
$\displaystyle \frac{\partial}{\partial t}$ with respect to the reference metric $\omega$, then we have the usual formula
\begin{equation}\label{mp44}
\frac{\partial}{\partial t}\sum_\alpha\int_{X_t}\theta_\alpha F_{\alpha, \delta}(z,t)e^{f_\alpha}\omega^n= \sum_\alpha\int_{X_t}v\big(\theta_\alpha F_{\alpha, \delta}(z,t)e^{f_\alpha}\big)\omega^n.
\end{equation}
The formula \eqref{mp41} shows that $\displaystyle \frac{\partial F_{\alpha, \delta}}{\partial t}$
converges to zero as $\delta\to 0$ because only the weights $\phi_j$ depend on $t$ and the coefficients $a_j$ are strictly smaller than 1.
Indeed, we have
\begin{equation}\label{pm01}
\frac{\partial F_{\alpha, \delta}}{\partial t}= \sum_j \frac{\delta^2e^{\phi_j}\partial_t\phi_j }
{\left(|z_{j}|^2+\delta^2e^{\phi_j}\right)^{1+a_j}}
\frac{a_j}{\prod_{i\neq j}\left(|z_{i}|^2+\delta^2e^{\phi_i}\right)^{a_i}}
\end{equation}
and our claim follows since $\displaystyle \int_{(\CC, 0)}\frac{\delta^2}{(|z|^2+ \delta^2)^{1+a}}d\lambda(z)\to 0$ as $\delta\to 0$ for any $a< 1$.

\noindent As for terms
involving $\displaystyle \frac{\partial F_{\alpha, \delta}}{\partial z_i}$ we infer the same conclusion (i.e. they tend to zero) by using
integration by parts as we explain next. The corresponding terms in \eqref{mp44} have the following shape
\begin{equation}\label{pm02}
\int_{X_t}\frac{\partial F_{\alpha, \delta}}{\partial z_i}(z) \tau_{\alpha}(z)d\lambda(z)
\end{equation}
where $\tau_{\alpha}$ is a smooth function with compact support in $W_\alpha\cap X_t$. The integral
\eqref{pm02} is equal to
\begin{equation}\label{pm03}
-\int_{X_t}\frac{\partial \tau_{\alpha}}{\partial z_i}(z) F_{\alpha, \delta}(z)d\lambda(z)
\end{equation}
and this tends to zero by dominated convergence.

\noindent The same type of arguments apply for the second order derivatives of $\Cdt$; the claim \eqref{mp40} follows.
\medskip

\noindent As $\vdj\to w$ in the $\mathcal C^{\infty}_{\rm loc}(X_t\ssm \Supp(B))$ topology when $j \to +\infty$, it follows from the identity \eqref{dbar} above that $w_{|X_t\ssm \Supp(B)}$ is holomorphic.
\end{proof}
\medskip

\noindent The next proposition is equally very important in the analysis of the uniformity properties of $(v_\delta)_{\delta> 0}$.

\begin{prop}
\label{conv2}
Let $t\in \Delta$ be fixed. Then the identity
\begin{equation}\label{mp50}
  \lim_{\delta \to 0}\left(c(\td)-\int_{X_t}c(\td) \td^n\right)=0
\end{equation}
holds on $X_t\ssm \Supp(B)$.
\end{prop}

\begin{proof}
Let $G_{\delta}:X_t\times X_t \to \mathbb R$ be the Green function of $(X_t, \td)$. Let $x\in X_t\ssm \Supp(B)$; by definition, one has
\begin{equation}
\label{green}
c(\td)(x)-\int_{X_t}c(\td) \td^n=\int_{X_t} -\Delta_{\td}c(\td) \cdotp G_{\delta}(x,\cdotp) \,\td^n
\end{equation}
Clearly, $\mathrm{Vol}(X_t, \td)=\int_{X_t} \td^n = \int_{X_t} \om^n=1$ is independent of $\delta$. Moreover, by \eqref{coniccontr}, there exists a constant $C_1>0$ independent of $\delta$ such that $\mathrm{diam}(X_t,\td) \le C_2$. Therefore, it follows from \cite[A.2]{Siu87} that
\begin{equation}
\label{lower}
G(x,y) \ge -C_2
\end{equation}
for some $C_2>0$ independent of $\delta$. Now recall that $G_{\delta}(x,y)=\int_{0}^{+\infty} G_{\delta}(x,y,s) ds$ where $G_{\delta}(x,y,s)$ satisfies
$$G_{\delta}(x,y,s) \le \begin{cases}
C_3 s^{-n}e^{-d_{\td}(x,y)/5s} & \mbox{if } 0<s<1\\
C_4 s^{-n} & \mbox{for any } 0<s<+\infty
\end{cases}$$
where $d_{\td}$ is the geodesic distance induced by $\td$ on $X_t$. This follows respectively by \cite[Thm.~16]{Davies} and \cite[p.139]{Siu87} \--- recall that the Ricci curvature of $\td$ is uniformly bounded below thanks to \eqref{coniccontr}.
Integrating the above inequalities, one gets
\begin{equation}
\label{upper}
G(x,y) \le C_3\,  d_{\td}(x,y)^{2-2n}
\end{equation}
for some uniform $C_3>0$. Let $I_{\delta}(x):= c(\td)(x)-\int_{X_t}c(\td) \td^n$, and let $C_4>0$ be large enough so that $\pm \Theta_{\delta} \le C_4 \om$. One has successively:
\begin{eqnarray*}
|I_{\delta}(x)| &=& \left|\int_{X_t} -\Delta_{\td}c(\td) \cdotp (G_{\delta}(x,\cdotp)+C_2) \,\td^n \right|\\
&\le &\int_{X_t}\left(\|\dbar \vd\|^2 +C_4(\sum_j\frac{\delta^2}{|s_j|^2+\delta^2})|\vd|_{\om}^2\right) \cdotp (G_{\delta}(x,\cdotp)+C_2)\,\td^n\\
&+ & \int_{X_t}\left(\sum_ja_j\delta^2\frac{\sqrt{-1}\la\partial s_j, \partial s_j\ra
  (\vd,\bar v_{\delta})}{(|s_j|^2+\delta^2)^2}\right) \cdotp (G_{\delta}(x,\cdotp)+C_2)\,\td^n\\
&\le &C_5 \int_{X_t}\left(\|\dbar \vd\|^2 +(\sum_j\frac{\delta^2}{|s_j|^2+\delta^2})|\vd|_{\om}^2\right) \cdotp d_{\td}(x,\cdotp)^{2-2n}\,\td^n\\
&+ & \int_{X_t}\left(\sum_ja_j\delta^2\frac{\sqrt{-1}\la\partial s_j, \partial s_j\ra
  (\vd,\bar v_{\delta})}{(|s_j|^2+\delta^2)^2}\right) \cdotp d_{\td}(x,\cdotp)^{2-2n}\,\td^n
\end{eqnarray*}
We claim that the right hand side converges to $0$ when $\delta\to 0$, uniformly on $x$ belonging to a fixed compact subset of $ X_t\ssm \Supp(B)$. To see this, it is enough to check that out of any sequence $\delta_j\to0$, one has $\lim_{j\to +\infty} I_{\delta_j}(x)=0$ uniformly on $x$, up to extracting a subsequence. Thanks to Lemma~\ref{conv}, one can assume that $v_{\delta_j}$ converges locally smoothly to a holomorphic vector field $w$ on $ X_t\ssm \Supp(B)$. Let us pick $\ep>0$.

By the estimates and observations above, one can find a small neigborhood $U_x \Subset X_t \ssm \Supp(B)$ and a constant $C=C(x)>0$ such that:
\begin{enumerate}
\item[(i)] $|\vd|^2_{\om} \le C$, $\|\dbar v_{\delta_j}\|^2\le \ep$, and $|s_j|^2 \ge C^{-1}$ hold on $U_x$ for any index $j$;
\item[(ii)] $\int_{U_x} d_{\td}(z,\cdotp)^{2-2n} \td^n \le C$;
\item[(iii)] $d_{\td}(z,w)^{2-2n} \le C$ for any $w\notin U_x$.
\end{enumerate}
The rest of the proof is easy: we split the integral into two pieces on $U_x$ and its complement.

\noindent $\bullet$ On the complement of $U_x$ we use the item (iii) so that we can replace the function $d_{\td}(x,\cdotp)^{2-2n}$ in the inequalities above by a constant independent of $\delta$. The proof of Proposition \ref{conv} shows that the integral of the remaining terms tends to 0 and $\delta\to 0$.

\noindent $\bullet$ On the set $U_x$ we are `far' from the support of $B$. Combined with the items (i) and (ii) above,
this finishes the proof of Proposition~\ref{conv2}.
\end{proof}
\medskip

\noindent In fact, Proposition~\ref{conv2} shows that the limit \eqref{mp50}
is uniform on compact sets contained in the complement of the divisor $B$.
We intend to couple this with the elliptic equation satisfied by $c(\td)$
in order to obtain bounds for the derivatives of this function in the fiber directions. To this end, we need the following statement.

\begin{prop}
\label{bounded}
There exists a constant $C>0$ independent of $\delta>0$ such that
$$\left|\int_{X_t}c(\td) \td^n\right| \le C $$
\end{prop}

\begin{proof} This statement can be seen as a by-product of the
  considerations in the article \cite[(5.3) \& Prop.~5.4]{Gue16}. Therefore we will
  content ourselves to highlight the main steps.

\noindent
  To start with, we recall that the normalization of $u_\delta$ is as follows
 \begin{equation}\label{mp51}
\int_{X_t}\ud \frac{\Omm}{\prod_j\left(|f_{j}|^2+\delta^2e^{\phi_j}\right)^{a_j}} =0
 \end{equation}
 and this can be re-written as
 \begin{equation}\label{mp52}
\int_{X_t}\ud e^{F_\delta}\omega_\delta^n =0
 \end{equation}
where $\omega_\delta$ is a metric with conic singularities on $X$, whose multiplicities along the components of $B$ are $1>b_j\geq \max(a_j, 1/2)$ (notations as in \eqref{coniccontr}). Note that $F_\delta$ in \eqref{mp52} has an explicit expression, being the log of the ratio $\displaystyle \frac{\td^n}{\omega_\delta^n}
$.

Let $V_\delta$ be the horizontal lift of $\displaystyle \frac{\partial}{\partial t}$ with respect to $\omega_\delta$. By applying the
$\displaystyle \frac{\partial^2}{\partial t\partial \ol t}$ operator in \eqref{mp52} we obtain
 \begin{eqnarray}\label{bigmess} \nonumber
   \int_{X_t}V_\delta \left(\ol V_\delta(\ud)\right) e^{F_\delta}\omega_\delta^n & = &
   -\int_{X_t}V_\delta(\ud) \ol V_\delta(F_\delta)e^{F_\delta}\omega_\delta^n-
   \int_{X_t}\ol V_\delta(\ud) V_\delta(F_\delta)e^{F_\delta}\omega_\delta^n\\
 & - & \int_{X_t}\ud V_\delta \left(\ol V_\delta(F_\delta)\right)e^{F_\delta}\omega_\delta^n- \int_{X_t}\ud \left|V_\delta(F_\delta)\right|^2e^{F_\delta}\omega_\delta^n
\end{eqnarray}

 \noindent Now the point is that, up to terms for which we have a uniform estimate already, the function $\displaystyle V_\delta \left(\ol V_\delta(\ud)\right)$ is ``the same'' as $c(\tau_\delta)$. Hence the absolute value of the
 lhs of \eqref{bigmess} is equivalent to $\displaystyle \left|\int_{X_t}c(\td) \td^n\right|$.

 \noindent
 The terms on the rhs of \eqref{bigmess} are
 uniformly bounded, as it is proved in the reference indicated at the beginning of the proof.
\end{proof}
\medskip

\noindent We can now prove that the vector field $v_\rho$ is holomorphic when restricted to the fibers of $p$.
\begin{coro}
\label{vhol}
Let $t\in \Delta$ be fixed. The family $(\vd|_{X_t})_{\delta >0}$ converges locally smoothly outside $\Supp(B)$ to the lift $v$ of $\frac{\d}{\d t}$ with respect to $\rho|_{X°\smallsetminus \Supp(B)}$. In particular, $v|_{X_t \ssm \Supp(B)}$ is holomorphic.
\end{coro}

\begin{proof}
  Combining Propositions~\ref{conv2} and \ref{bounded}, one sees that $c(\td)$ is locally uniformly bounded on $X_t \ssm \Supp(B)$. Given the elliptic equation satisfied by $c(\td)$, it implies local bound at any order (in the fiber directions).

Let $W\subset X$ be a coordinate open subset of $\cX$ such that $W\cap \Supp(B)= \emptyset$.  In local coordinates, this implies that
\begin{equation}\label{mp53}
\frac{\partial^2 u_\delta}{\partial t\partial \ol t}
\end{equation}
is bounded on $W$ by a constant independent of $\delta$. Since we already dispose of this type of bounds for any other mixed second order derivatives of $u_\delta$, we infer that we have
\begin{equation}\label{mp54}
  \left|\Delta^{\prime\prime}u_\delta\right|\leq C_W
\end{equation}
where $\Delta^{\prime\prime}$ is the Laplace operator corresponding to the flat metric on $W$ and $C_W$ is a constant independent of $\delta$.

  This implies that the global function $\ud$ admits $C^{1,\alpha}$ bounds locally on $\cX\ssm \Supp(B)$ for any $\alpha<1$. By Arzela-Ascoli theorem and Lemma~\ref{conv0}, it implies that $u_{\delta}$ converges to $\varphi$ in $C^{1,\alpha}_{\rm loc}(\cX\ssm \Supp(B))$. In particular, $\vp$ is differentiable in the $t$ variable outside $\Supp(B)$, and on this locus, $\d_t \vp_t = \lim \d_t u_{\delta}$ in the $\mathcal C^{\alpha}_{\rm loc}$ topology. Now, \eqref{dt} shows that the convergence actually takes places in $\mathcal C^{\infty}_{\rm loc}(X_t \ssm \Supp(B))$. In particular, outside $\Supp(B)$, $v_{\rho}|_{X_t}$ is the smooth limit of $\vd|_{X_t}$ when $\delta \to 0$. Corollary~\ref{vhol} is now a consequence of Proposition~\ref{conv}.
\end{proof}

\begin{coro}
\label{const}
Let $t\in \Delta$ be fixed. Then $dc(\td)|_{X_t}$ converges locally uniformly to $0$ on the compact subsets of $X_t \ssm \Supp(B)$.
\end{coro}

\begin{proof}
  Let $K\Subset X_t \ssm \Supp(B)$. By the proof of Corollary~\ref{vhol} and given \eqref{mp25}, $c(\td)|_K$ is bounded in $L^{\infty}$ norm hence in any $\mathcal C^k_{\rm loc}$ norm on $K$. This implies that family $dc(\td)|_{K}$ is relatively compact in the smooth topology, and the claim follows from Proposition~\ref{conv2}.
\end{proof}
\medskip

\noindent

\begin{lemm}
The vector field $v$ on $\cX\ssm \Supp(B)$ is holomorphic and extends across $\Supp(B)$.
\end{lemm}

\begin{proof}
  This first assertion follows from a simple computation in \cite[Lem.~2.5]{BoBJDG}.
  In our setting, this yields on $X_t\ssm \Supp(B_t)$:
\begin{equation}\label{mp55}
  \bar \d_t \vd \intprod {\td}=\bar \d c(\td)-i \td(\bar \d \vd, \bar{v}_{\delta})\end{equation}
As on $X_t\ssm \Supp(B_t)$, $\td$ and $\vd$ converge locally smoothly to $\rho$ and $v$ respectively, one deduces from Corollary~\ref{const} above that $v$ is holomorphic (hence smooth, too) in the $t$ variable as well, outside $\Supp(B)$.

For the second assertion, let us first observe that ${\td}^n\wedge idt\wedge d\bar t$ dominates a smooth volume form $dV$ on $\cX$. Therefore, it follows from \eqref{2} that
$$\int_{p^{-1}(U)\ssm \Supp(B)} |\vd|^2_{\om} \,  dV \le C$$
An application of Fatou lemma gives:
$$\int_{p^{-1}(U)\ssm \Supp(B)} |v|^2_{\om} \,  dV < +\infty$$
By Hartog's theorem, it follows that $v$ extends to a holomorphic vector field across $\Supp(B)$.
\end{proof}

\begin{lemm}
The vector field $v$ preserves $\rho$, hence its flow preserves $B$.
\end{lemm}

\begin{proof}
On $\cX\ssm \Supp(B)$, we obtain the equality
\begin{equation}\label{mp56}\mathcal{L}_{v}\rho= 0\end{equation}
as a consequence of \eqref{mp55}.

We show next that \eqref{mp56} extends in the sense of currents on $\cX$.
Indeed, if so then we claim that
the flow of $v$ produces the biholomorphic maps $F_t=X_0\to X_t$ such that
$F_0$ is the identity and such that
$F_t^*\omt=\om_0$. It is for this equality that we need \eqref{mp56}
to hold on $\cX$ in the sense of currents: it gives
\begin{equation}\label{mp57}
\frac{d}{dt}F_t^\star\omega_t= 0
\end{equation}
in weak sense on $\cX$, but this is enough to conclude that $F_t^*\omt=\om_0$.

If one pulls back the K\"ahler-Einstein equation satisfied by $\omt$ by $F_t$, one gets
$$\Ric F_t^{*}\om_t  = -F_t^*\om_t+F_t^*[B_t]$$ where $[B_t]=\sum_k a_k [B_{t,k}]$ if $B_{t,k}$ are the irreducible components of $\Supp(B)$. Because $F_t^*\om_t= \om_0$, we obtain $$F_t^*[B_t]= [B_0]$$ In particular, the local flow of $v$ preserves $\mathrm{Supp}(B)$.\\

Let us now prove that $v\intprod \rho$ is zero on $\cX$. First, let us observe that $\rho$ being a positive current, its coefficients are locally defined complex measures. We claim that these measures put no mass on $\Supp(B)$.

\noindent
Indeed, by e.g. \cite[Proposition 1.14]{Dem1} the "mixed terms" of $\rho$ are dominated by the trace of $\rho$ (the sum of the diagonal coefficients). Therefore everything boils down to showing that if $\om$ is a given smooth Kähler form on $\cX$, then the positive measure $\rho \wedge \om^n$ does not charge $\Supp(B)$. But it is easy to produce a family of cut-off function $\chi_{\delta}$ such that $\chi_{\delta}$ tends to the characteristic function of $\Supp(B)$ and such that $||\nabla{\om} \chi_{\delta}||_{L^2(\om^{n+1})}$ and $||\Delta_{\om} \chi_{\delta}||_{L^1(\om^{n+1})}$ tends to $0$. We refer to e.g. \cite[\S 9]{CGP} for this classic construction. Finally, let us introduce $\eta$ a smooth positive function with compact support on $\cX$. One can assume that on $\Supp(\eta)$, $\rho= \ddc \psi$ admits a local (bounded) potential. Performing an integration by parts, one obtains:
{\footnotesize
\begin{eqnarray*}
\int_{\cX}\eta \chi_{\delta} \,\rho \wedge \om^{n} &=& \int_{\cX} \eta \chi_{\delta} \,\ddc\psi \wedge \om^{n}\\
&=&\int_{\cX} \eta \psi\, \ddc\chi_{\delta}\wedge \om^n+\int_{\cX}\chi \psi \,\ddc \eta \wedge \om^n+ \int_{\cX}\psi \, d\eta \wedge d^c \chi_{\delta}\wedge \om^n \\
& \le & ||\psi||_{\infty} \left( ||\eta||_{\infty} \cdotp ||\Delta_{\om} \chi_{\delta}||_{L^1}+ ||\Delta \eta||_{\infty}\int_{\Supp(\chi_{\delta})}\om^{n+1}+||\nabla \eta||_{L^2} \cdotp ||\nabla \chi_{\delta}||_{L^2}\right)
\end{eqnarray*}}
which tends to $0$.
\smallskip

\noindent In conclusion, the coefficients of $\rho$ and hence those of $v \intprod \rho$  are complex measures which do not charge $B$. As $v \intprod \rho=0$ outside $\Supp(B)$, this identity extends across $\Supp(B)$, which is what we wanted to prove.
\end{proof}

If we sum up the results obtained so far, we can find near any $y\in Y°$ a sufficiently small polydisk $U\subset Y°$ with coordinates $(t_1, \ldots, t_m)$ centered around $y$ as well as holomorphic vector fields $v_1, \ldots, v_m$ on $p^{-1}(U)$ lifting $\frac{\d}{\d t_1}, \ldots, \frac{\d}{\d t_m}$ which are tangent to $\Supp(B)$. Up to shrinking $U$, one can assume that the flow of the vector fields $v_{\underline a}:=\sum a_i v_i$ for $\underline a = (a_1, \ldots, a_m) \in \mathbb D^m$ exists at least up to time one. Here $\mathbb D$ is the unit disk in $\mathbb C$. Then one has a holomorphic map $f:X_y \times \mathbb D^m \to p^{-1}(U)$ which sends $(x, \underline a)$ to $\phi_1^{\underline a}(x)$ where $(\phi_t^{\underline a})_t $ is the flow of $v_{\underline a}$. It is easy to see that $f$ is an isomorphism onto its image, cf e.g. \cite{KodM}.

To conclude the proof of Theorem \ref{foliationiso}, we need to show that $v_{\rho}$ extends across the singular locus of $p$ provided that $X$ is compact and $p$ is smooth in codimension one. The argument goes as follows.

\begin{proof}[End of the proof of Theorem~\ref{foliationiso}]
Let $n$ be the relative dimension of $p$ and let $m:=\dim Y$. Let $Y° \subset Y$ be the smooth locus of $p$, and let $X°:=p^{-1}(Y°)$.
Let $\Omega \in H^0 (X, m (K_{X/Y} +B))$.
Let $\rho =\omega + dd^c \psi$ be the positive current constructed in Theorem~\ref{thm:cyvar}, and let $y\in Y\ssm Y°$.

Let $x\in X$ be a generic point of $p^{-1} (y)$. Take a small neighborhood $U$ of $x$, and set $D:=p(U)$. As $p$ is smooth on codim $1$, $p$ is smooth on $U$.
We can thus fix a coordinate system $(\underline t, z_1, \ldots, z_n)$ of $U$, such that $\underline t$ represents the horizontal directions and $\frac{\partial}{\partial z_i}$ is in the fiber direction. The notation $\underline t$ means that $\underline t = (t_1, \ldots, t_m)$. There is a slight abuse of notation: the coordinate of the base is also $\underline t$. But as $p$ is smooth on $U$, we just mean that $p_* (\frac{\partial}{\partial t_i} ) =\frac{\partial}{\partial t_i}$, where the former is on $X$ and the later is on $Y$. Finally, we set $p^* (id\underline t \wedge \overline{ d \underline t}) :=\bigwedge_{k=1}^m id\underline t_k \wedge \overline{ d \underline t_k}$.

Let $v_k$ be the holomorphic vector field on $X° \cap p^{-1}(D)$ constructed in the proof of Theorem~\ref{foliationiso}, attached to $\frac{\d}{\d t_k}$, where $1\le k \le m$.

\begin{equation}
\label{mainequ}
\rho ^n \wedge \pdt = \frac{(\Omega \wedge \overline{\Omega})^\frac{1}{m}}{|f_B| ^2} \wedge \pdt \qquad\text{on } U .
\end{equation}

\noindent
 We know that $\iota_{v_k} \rho $ is proportional to  $d\bar t_k$, from which it follows that
\begin{equation}
\label{vvb}
\iota_{v_1, \bar v_1} \cdots \iota_{v_m,\bar v_m} (\rho^n \wedge p^* (id\underline t \wedge \overline{ d \underline t})  ) = \rho^n	
\end{equation}
Combining \eqref{mainequ} and \eqref{vvb}, one gets
$$\iota_{v_1, \bar v_1} \cdots \iota_{v_m,\bar v_m}  \left[\frac{(\Omega \wedge \overline{\Omega})^\frac{1}{m}}{|f_B| ^2} \wedge \pdt \right] = \rho^n$$
One can find a K\"ahler form $\om_X$ on $X$ such that $\frac{(\Omega \wedge \overline{\Omega})^\frac{1}{m}}{|f_B| ^2} \wedge \pdt \ge \om_X^{n+m}$. Given that $\om_X^m \wedge [\iota_{v_1, \bar v_1} \cdots \iota_{v_m,\bar v_m}(\om_X^{n+m})]= (\prod_k |v_k|^2_{\om_X})\cdotp  \om_X^{n+m}$ (maybe up to some constant), we eventually get that
$$\int_{U\cap X°}\left(\prod_{k=1}^m |v_k|^2_{\om_X}\right) \cdotp \om_X^{n+m} \le \int_{U\cap X°} \rho^n \wedge \om_X^m$$
and the right hand side is finite, dominated by $\int_{X}\la \rho^n \wedge \om_X^m \ra \le \{\om\}^{n}\cdotp \{\om_X\}^m$ by \cite[Prop.~1.6 \& 1.20]{BEGZ}, given that $\rho$ is a closed, positive current on $X$ in the cohomology class $\{\om\}$.

As $|v_k|^2_{\om_X}$ is uniformly bounded below by a positive constant on $p^{-1}(D)\cap X°$, one deduce that $v_k\in L^2(p^{-1}(D)\cap X°, \om_X)$.
By Riemann extension theorem the holomorphic vector fields $v_k$ extend to holomorphic vector fields on $p^{-1}(D)$ whose flow provide the expected trivialization. Indeed, the $v_k$ are tangent to $B$ on $X°$, hence they are tangent to $B$ everywhere by the assumptions in \ref{foliationiso}.
\end{proof}
\medskip

\noindent As application of Theorem~\ref{foliationiso} we can prove Corollary~\ref{bigdir}.

\begin{proof}[Proof of Corollary~\ref{bigdir}]
\label{bigdirp}
  Our proof follows the same line of arguments as in \cite{Kol87}.

  We proceed by contradiction: assume that $\cF_m$ is not big. In any case, this bundle
  can be endowed with a metric (used several times in the current subsection) with semi-positive curvature form denoted by $\theta$, and smooth on a Zariski open subset $V\subset Y$ as $B$ is generically transverse to the fibers. Then we claim that we have
\begin{equation}\label{app1}
\theta|_V^{\dim(Y)}= 0
\end{equation}
at each point of $V$. Indeed, if \eqref{app1} is not true, then there exists a point $y_0\in V$ such that all the eigenvalues of $\displaystyle \theta_{y_0}$ are strictly positive.
  By the singular version of holomorphic
  Morse inequalities (cf. \cite[Cor.~3.3]{Bou02}) this implies that $\cF_m$ is big, and we have assumed that this is not the case.

  It follows that the kernel of $\theta$ is non-trivial at each point of $V$. Since $\theta|_V$ is smooth and closed, locally near each point of $V$ its kernel defines a foliation whose leaves are analytic sets, cf \cite{Kol87} and the references therein. We choose a smooth holomorphic disk $\Delta$ contained in such a leaf; the restriction of $p$ to $p^{-1}(\Delta):= X_\Delta$ is a submersion, and the curvature of the direct image of the relative pluricanonical bundle is identically zero.
  By Theorem~\ref{foliationiso} we infer that the vector $v_\rho$ is holomorphic.
On the other hand, $\dbar v_\rho$ is a representative of the image of the tangent vector $\displaystyle \frac{\partial}{\partial t}\in T_\Delta$ by the map \eqref{ks}. Since by hypothesis this map is injective, we obtain a contradiction.
\end{proof}
\medskip

\noindent We finish the current section with the proof of Corollary~\ref{ambro}.

\begin{proof}[Proof of Corollary~\ref{ambro}]
\label{pageambro}
The statement \ref{part1cor} is a direct consequence of \cite{Gue16} applied to the right hand side term of the equality
\begin{equation}
\label{candecom}
- p^* (K_Y) =K_{X/Y} + (-K_X -B) +B.
\end{equation}
By hypothesis the class $-c_1(K_X+ B)$ is in the closure of the K\"ahler cone of $X$ and one can use \emph{loc. cit.}
\smallskip

\noindent
Given Theorem~\ref{foliationiso}, it would be enough to prove that $p$ is smooth in codimension one. We use the following elegant argument due to Q. Zhang, cf. \cite{Zha05}.
Assume that there exists some codimension one subvariety $D \subset X$ such that $p_* (D)$ is of codimension at least two.
Let $\tau : Y' \rightarrow Y$ be the composition of the blow-up of the closed analytic set $p_* (D)$ with a resolution of singularities of the resulting complex space. There exists an effective divisor $E_{Y'}$ whose support is contained
in the $\tau$-exceptional locus such that we have
$$K_{Y'} \sim E_{Y'}.$$
Let $p' : X' \rightarrow Y'$ be a resolution of indeterminacies of $X \dashrightarrow Y'$.
As $c_1 (K_X +B)=0$,
we have
$$ (p')^* (-K_{Y'}) +E_{X'} \equiv_{\Q} K_{X'/Y'} +B' ,$$
where $E_{X'}$ is supported in exceptional locus of $\pi: X' \rightarrow X$.
By \cite{Gue16}, $K_{X'/Y'} +B'$ is pseudo-effective. Therefore the direct image
$\pi_* ((p')^* (-K_{Y'}) +E_{X'}) = \pi_* ( -E_{Y'})$ is pseudo-effective as well.
However by construction we have $\pi_* (E_{Y'}) \geq [D]$, and we obtain a contradiction.

We prove next that the map $p$ is reduced in codimension one.
Let $E \subset Y$ be a divisor. Its $p$-inverse image can be written as
$$p^{-1} (E) = \sum_i a_i [D_i]$$
where $D_i\subset X$ are irreducible divisors.
It is well know that (cf. \cite[Thm.~2.4]{CP17} or also \cite{Taka16})
$$K_{X/Y} +B \geq \sum_i (a_i-1)_+ \cdot [D_i] ,$$
where $(a_i-1)_+ := \max \{a_i -1 ,0\}$.

Therefore we must have $a_i =1$ for every $i$, since
by assumption $K_{X/Y} +B \equiv_{\Q} 0$.
Corollary \ref{ambro}
is proved.
\end{proof}


\subsection{Log abundance in the Kähler setting}

In this section, we briefly explain how to prove the log abundance for klt Kähler pairs $(X,B)$ such that $B$ has snc support. This is based on the following lemma, which is a consequence of \cite{Bud09} and \cite[Cor 1.4]{WangB} (cf. also \cite[Lem.1.1]{CKP} and \cite{CT11} and the references therein). For the reader's convenience, we recall briefly the proof here. \footnote{We would like to thank Botong Wang for telling us the following nice application of his result.} 

After this paper was written, J. Wang \cite[Thm.~D]{JWang19} proved a slightly more general case of Corollary~\ref{log-ab} below using similar arguments. 

\begin{lemm}
\label{lemm:log-ab}
Let $X$ be a compact K\"ahler manifold and let $\Delta =\sum a_i B_i$ be an effective klt $\mathbb{Q}$-divisor with simple normal crossing support.
Assume that $\Delta \sim_{\mathbb Q} L_1$ for some $L_1 \in \Pic (X)$.
For each integer $k \geq 0$, define $L_k := k L_1  - \lfloor k \Delta \rfloor$.
Then for each $k, i$ and $q$, the set
$$V_i ^q ( L_k) =\{\lambda\in \Pic^\circ (X) ; \, h^q (X, K_X + L_k +\lambda) \geq i\}$$
is a finite union of translates of complex subtori of $\Pic^\circ (X)$ by torsion points.
\end{lemm}

\begin{proof}
Let $N$ be the minimal number such that $N \cdot a_i \in \mathbb{N}$
for every $i$. Let $\sigma :\widetilde{X} \rightarrow X$ be the $N$-cyclic cover of $L_1$ along the canonical section of $N L_1$. One can check that $\wt X$ has analytic quotient singularities \cite[Lem.~2]{Viehweg77}, hence rational singularities by e.g. \cite[Prop.~4.1]{Burns}. This implies in turn that for any resolution $\pi:\widehat{X} \rightarrow \widetilde{X}$, one has $\pi_*\mathcal O_{\widehat X}(K_{\widehat X}) =\mathcal O_{\wt X}(K_{\wt X})$ thanks to e.g. \cite[Thm.~5.10]{KM} and $R^i\pi_*\mathcal O_{\widehat X}(K_{\widehat X})=0$ for $i>0$ by Grauert-Riemenschneider vanishing. Moreover, one has $\sigma_*\mathcal O_{\wt X}(K_{\wt X})=\mathcal O_X( K_X)\otimes \bigoplus_{k=0}^{N-1}L_k$ and $R^i\sigma_*\mathcal O_{\wt X}(K_{\widetilde X})=0$ for $i>0$ since $\sigma$ is finite. Therefore, if we define $f:=\sigma \circ \pi: \widehat{X} \rightarrow X$, we have
\begin{equation}\label{splitdecom}
H^q (\widehat{X}, K_{\widehat{X}} + f^* \lambda) \simeq \bigoplus_{k=0}^{N-1} H^q (X, K_X +L_k +\lambda)
\end{equation}
for any line bundle $\lambda$ on $X$.

\medskip

Let $g: \Pic^\circ (X) \rightarrow \Pic^\circ (\widehat{X})$ be the natural morphism induced by $f$ and set
$$V_i ^q (f) := \{\rho\in \Pic^\circ (X) ; \, h^q ( \widehat{X} , K_{\widehat{X}} + f^* \rho)   \geq i \}$$
and
$$V_i ^q := \{\rho\in \Pic^\circ (\widehat{X}), h^q (\widehat{X} ; \, K_{\widehat{X}} +\rho)  \geq i \} .$$
Then we have
\begin{equation}\label{pulba}
V_i ^q (f) = g^{-1} (V_i ^q ) .
\end{equation}
Thanks to \cite{WangB}, $V_i ^q$ is a finite union of torsion translates of complex subtori of $\Pic^\circ (\widehat{X})$.
Together with \eqref{pulba}, this shows that $V_i ^q (f)$ has the same structure.
Thanks to \eqref{splitdecom}, we have
\begin{equation}\label{pulba1}
V_i ^q (f) = \cup_{i_0 +\cdots + i_{N-1} =i} \big( \cap_{k=0}^{N-1} V_{i_k} ^q (L_k)\big) ,
\end{equation}
where $V_i ^q (L_k) := \{\rho\in \Pic^\circ (X) , h^q (X , K_X + L_k + \rho)   \geq i \}$.
As $V_i ^q (f)$ is the finite union of torsion translates of complex subtori, we get from  \eqref{pulba1} that  $V_i ^q (L_k)$ has the same structure, cf \cite[Lemma 1.1]{CKP}.
\end{proof}

\begin{coro}
\label{log-ab}
Let $(X,\Delta)$ be a klt pair where $X$ is compact Kähler and $\Delta =\sum a_i B_i$ is an effective $\mathbb{Q}$-divisor.
If  $c_1 (K_X +\Delta) =0 \in H^{1,1} (X, \mathbb{Q})$, then $K_X +\Delta$ is $\mathbb{Q}$-effective.
\end{coro}

\begin{proof}
Let $ \pi: X' \rightarrow X$ be a log resolution of $(X,\Delta)$. Since $\Pic^\circ (X')$ is a torus and $c_1(K_X+\Delta)=0$, we can find $L\in \mathrm{Pic}^\circ(X')$ such that $\pi^*(K_X +\Delta) \sim_{\mathbb Q} L $. We can also find a klt divisor $\Delta'$ on $X'$ with normal crossing support such that
$$ K_{X'} +\Delta' \equiv_{\mathbb{Q}} \pi^* (K_X +\Delta) +E$$
for some $\mathbb{Q}$-effective divisor $E$ supported in the exceptional locus of $\pi$ having no common component with $\Delta'$. Let $m\ge 1$ be the smallest integer such that $mE$ has integral coefficients. In particular, $m(K_{X'}+\Delta')$ is equivalent to some line bundle on $X'$ by the formula above. Using the identity
$$m(K_{X'}+\Delta')=K_{X'}+\underbrace{\Big(\Delta'+\frac{m-1}m \{E\}\Big)}_{=:\Delta^+}+(m-1)(L+\lfloor E\rfloor)$$
we get a pair $(X',\Delta^+)$  such that
\begin{enumerate}
\item[$\bullet$] $\Delta^+$ has snc support and coefficients in $(0,1)\cap \mathbb Q$.  
\item[$\bullet$] $\Delta^+\sim M$ for some line bundle $M$ on $X'$. 
\item[$\bullet$] $K_{X'}+\Delta^++\rho$ is effective for some $\rho \in \mathrm{Pic}^\circ(X')$. 
\end{enumerate}
The first two properties are obvious, and the third follows from the identity $K_{X'}+\Delta^+-L=mE-(m-1)\lfloor E\rfloor$. By applying Lemma~\ref{lemm:log-ab} to $K_{X'} + \Delta^+$, we can assume that $\rho$ is torsion hence $h^0(X',r(K_{X'} + \Delta^+))\ge 1$ for some integer $r\ge 1$ that we can choose so that $m|r$. By doing so, one can ensure that $r(K_{X'}+\Delta^+)=\pi^*(r(K_X+\Delta))+F$ for some effective, integral $\pi$-exceptional divisor $F$. This implies that $h^0(X,r(K_X+\Delta)) \neq 0$. The Corollary is proved.
\end{proof}

\newpage

\section{Transverse regularity of singular Monge-Ampère equations}
\label{sec:transverse}
\noindent In this section our main goal is to prove Theorem \ref{thmc}. This will be achieved as a consequence of a few intermediate results which we state in a general setting.

The main source of difficulties in the proof of \ref{thmc} arise from the
fact that the set of base points of pluricanonical sections may be non-zero. The determinant of the metric adapted to this geometric setting \emph{vanishes} along the said base points so in particular the Ricci curvature of this metric is not bounded from below.
Unfortunately under these circumstances we were not able to obtain a complete analogue of
the Sobolev and Poincar\'e inequalities (which are needed for the study of the regularity properties of Monge-Amp\`ere equations). We will therefore start this section with a weak version of these results.

\subsection{Weak Sobolev and Poincar\'e inequalities}

In this section we will derive a version of the usual Poincar\'e and Sobolev type inequalities
which are needed in our context. As it is well known, they are playing a crucial role in the regularity questions for the Monge-Amp\`ere equations.
The set-up is as follows: let $(X, \omega)$ be a compact K\"ahler manifold of dimension $n$, and let
\begin{equation}\label{wsp1}
E:= \sum_{\alpha\in I} e_\alpha E_\alpha\quad B:= \sum_{\beta\in J} b_\beta B_\beta
\end{equation}
be two effective divisors on $X$ without common components,
such that $e_\alpha\in \Q_+, b_\beta\in [0, 1[$ and such that the support of $E+ B$ is snc.
We assume that the manifold $X$ is covered by a fixed family of coordinate sets $(\Omega_j)_j$
such that
\begin{equation}\label{E1}
\Omega_j\cap \Supp(E+ B)= (z_j^1\dots z_j^d= 0)
\end{equation}
where $(z_j)$ are coordinates on $\Omega_j$.

\noindent Let $\sigma_i, s_i$ be the canonical section of the
Hermitian bundle $\left(\O(E_i), h_i\right)$ and $\left(\O(B_i), g_i\right)$ respectively, where $h_i$ and $g_i$ are non-singular reference metrics.
For each positive $\epsilon\geq 0$ and each multi-index $q$ we introduce the following volume element
\begin{equation}
d\mu^{(\ep)}_q:= \frac{\prod_{\alpha\in I}(\ep^2+ |\sigma_\alpha|^2)^{q_\alpha}}{\prod_{\beta\in J}(\ep^2+ |s_\beta|^2)^{b_\beta}}dV_\omega
\end{equation}
where $dV_\omega$ is the volume element corresponding to the reference metric $\omega$.
Also, for each positive real number $p\leq 2$ we define
the multi-index $q_p$ whose components are
\begin{equation}
\left(1-\frac{p}{2}\right)q_\alpha.
\end{equation}

\noindent Then we have the following statements.

\begin{prop}
\label{WS} There exists a constant $C> 0$ independent of $\ep$ (but depending on everything else)
such that for every smooth function $f$ on $X$ we have
\begin{equation}\label{WS1}
\left(\int_X|f|^{\frac{2np}{2n-p}}d\mu^{(\ep)}_q\right)^{\frac{2n-p}{2np}}\leq
C\left(\int_X|\nabla_\ep f|^pd\mu^{(\ep)}_{q_p} +  \int_X|f|^pd\mu^{(\ep)}_{q_p} \right)^{\frac{1}{p}}
\end{equation}
where $1\leq p< 2$ is any real number, and
the gradient $\nabla_\ep$ corresponds to the
$\ep$-regularization of a fixed metric with conic singularities along the divisor
$\displaystyle \sum_{\beta\in J}b_\beta B_\beta$.
\end{prop}

As we can see, there is an important difference
between the Proposition \ref{WS} and the standard \emph{weighted Sobolev inequalities}: the volume element in the left hand side of \eqref{WS1} is not the same as the one in the right hand site term.
\smallskip

\noindent In a similar vein, we have the next version of the Poincar\'e inequality.

\begin{prop}
\label{WP} There exists a constant $C> 0$ as above such that for any smooth function $f$ on $X$
we have
\begin{equation}\label{WP1}
\int_X\left|f- \VM_\mu(f)\right|^p d\mu^{(\ep)}_q\leq C\int_X|\nabla_\ep f|^p d\mu^{(\ep)}_{q_p}
\end{equation}
where $p\geq 1$ is a real number, and where we use the notation
\begin{equation}\label{WP2}
\VM_\mu(f):= \int_Xfd\mu^{(\ep)}_q.
\end{equation}
\end{prop}
\smallskip

\noindent We first prove the statement \ref{WS}; the arguments which will follow have
been ``borrowed" from the book \cite[Chap.~15]{juha}.
\begin{proof} [Proof of Proposition \ref{WS}] We first assume that $B=0$ because the arguments for the general case are practically identical.

A first remark is that it is enough to consider the local version of the statement, as follows.
Let $\Omega$ be one of the domains covering $(X, E)$ as mentioned in \eqref{E1}; we denote by
$(z_1,\dots , z_n)$ the corresponding coordinate system.  We will assume that we have
\begin{equation}
\Omega= \prod_j(|z_j|< 1)
\end{equation}
and that the function $f$ has compact support in $\Omega$.

In terms of this local setting, the quantity to be evaluated becomes

\begin{equation}
\int_\Omega|f|^{\frac{2np}{2n-p}}\prod_{\alpha=1}^d(\ep^2+ |z_\alpha|^2)^{q_\alpha}d\lambda
\end{equation}
(since $b_i=0$).
Let $\B:= (|t|< 1)\subset \CC$ be the unit disk in the complex plane. We consider the function
\begin{equation}\label{wsp2}
F_\ep(t)= \frac{(\ep^2+ |t|^2)^{q/2}}{(1+ \ep^2)^{q/2}}t
\end{equation}
where $q>0$ is a real number. It turns out that $F_\ep$ is a diffeomorphism and the square
of the absolute value of its Jacobian $\displaystyle dF_\ep\wedge d\overline F_\ep$ verifies the
inequality
\begin{equation}
C^{-1}(\ep^2+ |t|^2)^{q}\leq \frac{dF_\ep\wedge d\overline F_\ep}{dt\wedge d\overline t}\leq
C(\ep^2+ |t|^2)^{q}
\end{equation}
where $C$ is a constant independent of $\ep$ (it can be explicitly computed). Let $G_\ep$ be the
inverse of $F_\ep$. The implicit function theorem shows that we have
\begin{equation}
|dG_\ep(t) |\leq \frac{C}{(\ep^2+ |t|^2)^{q/2}}.
\end{equation}

\noindent By the change of variables formula we have
\begin{equation}\label{wsp3}
\int_\Omega|f(z)|^{\frac{2np}{2n-p}}\prod_{\alpha=1}^d(\ep^2+ |z_\alpha|^2)^{q_\alpha}d\lambda\leq C
\int_\Omega|\widetilde f(w)|^{\frac{2np}{2n-p}}d\lambda(w)
\end{equation}
where by definition we set
\begin{equation}\label{WS2}
\widetilde f(w):= f\big(G_\ep(w_1),\dots , G_\ep(w_d), w_{d+1},\dots, w_n\big);
\end{equation}
it is a function defined on the ``same" poly-disk $\Omega$, and it has compact
support.

Therefore, by the usual version of the Sobolev inequality we obtain

\begin{equation}\label{}
\left(\int_\Omega|\widetilde f(w)|^{\frac{2np}{2n-p}}d\lambda(w)\right)^{\frac{2n-p}{2n}}
\leq C\int_\Omega|\nabla \widetilde f(w)|^p d\lambda(w).
\end{equation}
We use the relation \eqref{WS2}, together with the change of coordinates
$\displaystyle w_\alpha= F_\ep(z_\alpha)$ for $\alpha= 1,\dots d$ and we infer that we have
\begin{equation}
\int_\Omega|\nabla \widetilde f(w)|^p d\lambda(w)\leq C
\int_\Omega|\nabla f(z)|^{p}\prod_{\alpha=1}^d(\ep^2+ |z_\alpha|^2)^{q_\alpha(1- \frac{p}{2})}d\lambda.
\end{equation}
In conclusion we have
\begin{equation}
\left(\int_\Omega|f(z)|^{\frac{2np}{2n-p}}
\prod_{\alpha=1}^d(\ep^2+ |z_\alpha|^2)^{q_\alpha}d\lambda\right)^{\frac{2n-p}{2n}}\leq C
\int_\Omega|\nabla f(z)|^{p}\prod_{\alpha=1}^d(\ep^2+ |z_\alpha|^2)^{q_\alpha(1- \frac{p}{2})}d\lambda.
\end{equation}
that is to say, we have established the local version of the inequality
\ref{WS}. The general case follows by a partition of unit argument which we skip.
\end{proof}
\medskip

\noindent The same scheme of proof applies to Proposition \ref{WP}: we will
first show that the local version of this statement holds by using a
change of coordinates and the classical version of Poincar\'e inequality,
and then we show that the global version \eqref{WP1} is true by a
well-chosen covering of $X$.
\begin{proof} The inequality \eqref{WP1} is easily seen to follow provided that
we are able to establish
the following relation

\begin{equation}\label{eq40}
  \int_{X\times X}|f(x)- f(y)|^p d\mu_q^{(\ep)}(x)d\mu_q^{(\ep)}(y)\leq C\int_X|\nabla f|^p
  \mu_{q_p}^{(\ep)}
\end{equation}
for any $1\leq p\leq 2$. This is very elementary and we will not provide any additional explanation.

Assume that we have a covering of $X$
\begin{equation}\label{eq41}
X= \bigcup_i U_i
\end{equation}
where each $U_i$ is a coordinate open set. In order to obtain a bound as
in \eqref{eq40}, it would be enough to analyze the quantities
\begin{equation}\label{eq42}
 \int_{U_i\times U_j}|f(x)- f(y)|^p d\mu_q^{(\ep)}(x)d\mu_q^{(\ep)}(y)
\end{equation}
for each couple of indexes $i, j$ which is what we do next.
\smallskip

\noindent To start with, let $\Omega$ be one of the coordinate sets $U_i$;
we will show next that the following local version of \eqref{WP1} holds true
\begin{equation}\label{eq43}
  \int_{\Omega\times \Omega}|f(x)- f(y)|^p d\mu_q^{(\ep)}(x)d\mu_q^{(\ep)}(y)\leq C\int_\Omega|\nabla f|^p \mu_{q_\alpha}^{(\ep)}.
\end{equation}
We proceed as in the previous proof: we have
\begin{equation}\label{eq44}
\int_{\Omega\times \Omega}|f(x)- f(y)|^p d\mu_q^{(\ep)}(x)d\mu_q^{(\ep)}(y)\leq C
\int_{\Omega\times\Omega}\left|\widetilde f(z)- \widetilde f(w)\right|^pd\lambda(z,w)
\end{equation}
by a change of coordinates as indicated in \eqref{wsp2}. Now we have
\begin{equation}\label{eq45}
\widetilde f(z)- \widetilde f(w)= \int_0^1\frac{d}{dt}\widetilde f\left((1-t)z+ tw)\right) dt
\end{equation}
and it follows that we have
\begin{equation}\label{eq46}
  \int_{\Omega\times \Omega}\left|\widetilde f(z)- \widetilde f(w)\right|^p
  d\lambda(z,w)\leq
C\int_0^1dt \int_{\Omega\times \Omega}|\nabla \widetilde f\left((1-t)z+ tw)\right)|^p d\lambda(z,w),
\end{equation}
where the constant $C>0$ in \eqref{eq46} depends on the diameter of $\Omega$
measured with respect to the Euclidean metric.

Then we invoke the usual trick: we split the integral above in two --the
first part is
as follows
\begin{equation}\label{eq47}
\int_0^{1/2}dt \int_{\Omega\times \Omega}|\nabla \widetilde f\left((1-t)z+ tw)\right)|^p d\lambda(z,w)\leq C \int_{\Omega}|\nabla \widetilde f\left(z)\right)|^p d\lambda(z)
\end{equation}
where up to a numerical constant, $C$ in \eqref{eq47} only depends on the volume of $\Omega$. We have a similar estimate for the integral corresponding to
the interval [1/2, 1], so all in all we infer
\begin{equation}\label{eq48}
\int_{\Omega\times \Omega}\left|\widetilde f(z)- \widetilde f(w)\right|^p
d\lambda(z,w)\leq
C \int_{\Omega}|\nabla \widetilde f\left(z)\right)|^p d\lambda(z).
\end{equation}
Changing the coordinates back, together with the considerations in the proof of weak Sobolev inequality
show that \eqref{eq43} is proved.
\smallskip

\noindent The general case follows by choosing a covering $(U_j)$
of $X$, such that the following properties
are satisfied.
\begin{enumerate}

\item[(1)] If $U_p\cap U_q\neq \emptyset$ and if at least one of them intersects the support of the divisor $E$, then the union $U_p\cup U_q$ is contained in a coordinate set endowed with coordinates adapted to $(X, E)$ (as in the beginning of this section).
\smallskip

\item[(2)] If $U_p\cap U_q\neq \emptyset$ and if neither of $U_i$ or $U_j$ intersects $\Supp (E)$, then
  the union $U_p\cup U_q$ is contained in a coordinate ball which is disjoint of $\Supp (E)$.\smallskip

\item[(3)] The $ d\mu_q^{(\ep)}$-volume of the coordinate sets containing $U_i\cup U_j$ in (1) and (2) is bounded from above and below by constants which are independent of $\varepsilon$.    

\end{enumerate}

\noindent It is clear that such cover exists, and we fix one denoted by
$\Lambda$ for the rest of the proof. Note that this cover is independent of
$\varepsilon$.
Next, given any couple
$\displaystyle U_i, U_j$ of sets belonging to $\Lambda$, we consider a
collection
\begin{equation}\label{eq49}
\Xi_{ij}= \left(\Omega_1,\Omega_2,\dots ,\Omega_N\right)
\end{equation}
of elements of $\Lambda$ such that the following properties are
verified.
\begin{enumerate}

\item[(a)] We have $\Omega_1:= U_i$ and $\Omega_N:= U_j$, and all of the intermediate
$\Omega$'s are elements of $\Lambda$.
\smallskip

\item[(b)] For any $r=1,\dots N-1$ we have $\displaystyle
\Omega_r\cap \Omega_{r+1}\neq\emptyset$.
\end{enumerate}

\noindent Again, there are many choices for such $\Xi_{ij}$, but we just
pick one of them for each pair of indexes $(i, j)$.
\smallskip

We are now ready to analyze the quantities \eqref{eq42}: for each
couple $(i, j)$ we consider the collection $\Xi_{ij}$. Given
\begin{equation}\label{eq50}
  (x_1,\dots, x_N)\in \Omega_1\times\dots\times \Omega_N
\end{equation}
we have
\begin{equation}\label{eq51}
  |f(x_1)- f(x_N)|^\alpha\leq C\sum_q|f(x_q)- f(x_{q+1})|^\alpha
\end{equation}
for some numerical constant $C>0$.

\noindent We consider now the following expression
\begin{equation}\label{eq52}
\int_{\Omega_1\times\dots\times \Omega_N} |f(x_1)- f(x_N)|^p d\mu_q^{(\ep)}(x_1)\dots
d\mu_q^{(\ep)}(x_N);
\end{equation}
on one hand, up to a constant this is simply \eqref{eq42}. One the other hand \eqref{eq52}
is bounded from above by
\begin{equation}\label{eq53}
 C\sum_q\int_{\Omega_1\times\dots\times \Omega_N} |f(x_q)- f(x_{q_1})|^p d\mu_q^{(\ep)}(x_1)\dots
d\mu_q^{(\ep)}(x_N)
\end{equation}
The last observation is that each term of the sum \eqref{eq53} is of type \eqref{eq43}
--here we are using the properties (1)-(3) and (a), (b) above--
for which we have already shown the desired Poincar\'e inequality.
This ends the proof of the case $B=0$.
\medskip

\noindent We will not detail the proof of the general statement, because
the arguments are identical to the ones already given. The
only change is that we will work with geodesics with respect to
the model conic metric
\begin{equation}\label{wsp4}
  \sqrt{-1}\sum_{\alpha\in J}\frac{dz_\alpha\wedge d\ol z_\alpha}{(\ep^2+ |z_\alpha|^2)^{b_\alpha}}+ \sqrt{-1}\sum_{\alpha\not\in J}dz_\alpha\wedge d\ol z_\alpha
\end{equation}
instead of straight lines $(1-t)x+ ty$. The same proof works because the Ricci curvature
of the metric \eqref{wsp4} is bounded from below by some constant independent of $\ep$. For a complete treatment of this point we refer to \cite{SC}, pages 177-179.
\end{proof}

\subsection{Lie derivative of fiberwise Monge-Amp\`ere Equations}
In this subsection we consider the restriction of our initial family $p$ to a generic disk contained in the base, together with a family of Monge-Amp\`ere equations of its fibers. Let $\DD\subset Y$ be a one-dimensional germ of submanifold  contained in a coordinate set of $Y$, and let $\cX:= p^{-1}(\DD)$ (notations as in Theorem \ref{thmc}).

The resulting map $p:\cX\to \DD$ will be a proper submersion, provided that $\DD$ is generic. We recall that the
total space $(\cX, \omega)$ of $p$ is a K\"ahler manifold. We denote by $t$ a coordinate on the unit disk $\DD$, and let
\begin{equation}\label{eq1}
v= \frac{\partial}{\partial t}+ v^\alpha \frac{\partial}{\partial z^\alpha}
\end{equation}
be the local expression of a smooth vector field which projects into $\displaystyle \frac{\partial}{\partial t}$.

Another piece of data is the following fiberwise Monge-Amp\`ere equation
\begin{equation}\label{eq2}
(\omega+ dd^c\varphi)^n= e^{\lambda \varphi+ f}\omega^n
\end{equation}
on each $\cX_t$. Here $\lambda\geq 0$ is a positive real number, and $f$ is a smooth function on $\cX$.
We can write this globally as follows
\begin{equation}\label{eq3}
(\omega+ dd^c\varphi)^n\wedge \sqrt{-1}dt\wedge d\overline t=  e^{\lambda \varphi+ f}\omega^n \wedge \sqrt{-1}dt\wedge d\overline t
\end{equation}
on $\cX$, where the meaning of $dd^c$ and of $\varphi$ is not the same as in \eqref{eq2}, but...
\smallskip

We take the Lie derivative $\cL_v$ of \eqref{eq3} with respect to the vector field $v$, and then restrict to a fiber
$\cX_t$. The Lie derivative of the left-hand side term of \eqref{eq3} equals

\begin{equation}\label{eq7}
n\cL_v(\omega+ dd^c\varphi)\wedge (\omega+ dd^c\varphi)^{n-1}\wedge \sqrt{-1}dt\wedge d\overline t
\end{equation}
because we have $\cL_v\left(\sqrt{-1}dt\wedge d\overline t\right)= 0$, given the expression
\eqref{eq1}.

The form $\omega+ dd^c\varphi$ is closed, hence by Cartan formula we have
\begin{equation}\label{eq8}
\cL_v(\omega+ dd^c\varphi)= d\big( i_v\cdot (\omega+ dd^c\varphi)\big)
\end{equation}
where $\displaystyle {i_v}\cdot \omega$ is the contraction of $\omega$ with respect to the vector field $v$. We evaluate next the quantity
\begin{equation}\label{eq9}
d\big({i_v}\cdot dd^c\varphi\big)\wedge \sqrt{-1}dt\wedge d\overline t
\end{equation}
by a point-wise computation, as follows. With respect to the local coordinates as in \eqref{eq1}, we write
\begin{equation}\label{eq10}
dd^c\varphi= \varphi_{t\ol t}\sqrt{-1}dt\wedge d\overline t+ \varphi_{t\ol \alpha}
\sqrt{-1}dt\wedge dz^{\ol \alpha}+ \varphi_{\beta \ol t}
\sqrt{-1}dz^\beta\wedge d\ol t+  \varphi_{\beta \ol \alpha}\sqrt{-1}dz^\beta\wedge dz^{\ol \alpha};
\end{equation}
in the expression above we are using the Einstein convention. Then we have
\begin{equation}\label{eq11}
d\big({i_v}\cdot dd^c\varphi\big)\equiv \big(\varphi_{t\beta\ol \alpha} + \varphi_{\gamma\beta\ol \alpha}
v^\gamma+  \varphi_{\gamma\ol \alpha}v^\gamma_\beta\big)dz^\beta\wedge dz^{\ol \alpha}
\end{equation}
where $\equiv$ means that we are only consider the terms of (1,1)-type which do not contain $dt$
or its conjugate.

On the other hand, the coefficients of the Hessian of the function
\begin{equation}\label{eq12}
v(\varphi)= \varphi_t+ \varphi_\gamma v^\gamma
\end{equation}
in the fibers direction are equal to
\begin{equation}\label{eq13}
v(\varphi)_{\beta\ol \alpha}= \varphi_{t\beta\ol \alpha} + \varphi_{\gamma\beta\ol \alpha}
v^\gamma+  \varphi_{\gamma\ol \alpha}v^\gamma_\beta+ \varphi_{\gamma\beta}v^\gamma_{\ol\alpha}
+ \varphi_\gamma v^\gamma_{\beta\ol\alpha}.
\end{equation}
The first three terms in the expression \eqref{eq13} are identical to those in \eqref{eq11}. As for the
last two terms, they can be expressed intrinsically as follows
\begin{equation}\label{eq14}
(\varphi_{\gamma\beta}v^\gamma_{\ol\alpha}
+ \varphi_\gamma v^\gamma_{\beta\ol\alpha})dz^\beta\wedge dz^{\ol \alpha}=
\partial\left(\dbar v\cdot \varphi\right).
\end{equation}
Here $\dbar v$ is a $(0,1)$-form with values in $\displaystyle T_{\cX_t}$ and then
$\dbar v\cdot \varphi$
is a form of (0,1) type on $\cX_t$.
\smallskip

On the other hand, if we denote by $\Delta_{\varphi}= \Tr_\varphi\sqrt{-1}\ddbar$ the Laplace operator corresponding to
the metric $\omega_\varphi:= \omega+ dd^c\varphi$ on the fibers of $p$, then we can rewrite the
equation \eqref{eq8} as follows
\begin{equation}\label{eq4}
\left(\Delta_{\varphi} v(\varphi) - \Tr_\varphi \partial \big(\dbar v\cdot \varphi\big)+
\Psi_{\varphi, v}\right)\omega_\varphi^n\wedge \sqrt{-1}dt\wedge d\overline t.
\end{equation}
In the expression \eqref{eq4} we denote by $\Tr_\varphi$ the trace with respect to
$\omega_\varphi$ on $\cX_t$, and we denote by $\Psi_{\varphi, v}$ the function on $\cX$ such that
the equality
\begin{equation}\label{eq15}
\Psi_{\varphi, v}\,\omega_\varphi^n\wedge \sqrt{-1}dt\wedge d\overline t= \cL_v(\omega)\wedge
\omega_\varphi^{n-1}\wedge \sqrt{-1}dt\wedge d\overline t
\end{equation}
holds on $\cX$.
\medskip

As for the right hand side of \eqref{eq3}, the expression of the Lie derivative
 reads as follows
\begin{equation}\label{eq16}
\left(\lambda v(\varphi)+ v(f)+ \Psi_v\right)\omega_\varphi^{n}\wedge \sqrt{-1}dt\wedge d\overline t
\end{equation}
where -as before- the function $\Psi_v$ is defined by the equality
\begin{equation}\label{eq17}
\Psi_{v}\,\omega^n\wedge \sqrt{-1}dt\wedge d\overline t= \cL_v(\omega)\wedge
\omega^{n-1}\wedge \sqrt{-1}dt\wedge d\overline t.
\end{equation}
\medskip

\noindent In conclusion, for each $t\in \DD$ we obtain the equality
\begin{equation}\label{eq18}
\Delta_{\varphi} v(\varphi) - \Tr_\varphi \partial \big(\dbar v\cdot \varphi\big)+
\Psi_{\varphi, v}= \lambda v(\varphi)+ v(f)+ \Psi_v
\end{equation}
which is the identity we intended to obtain in this subsection.\qed

\subsection{Regularity in transverse directions}

In this section we will apply the results above in order to
analyze the transversal regularity of the solution of the equation
\begin{equation}\label{eq19}
(\omega+ dd^c\varphi_t)^n= e^{\lambda \varphi+ f}\frac{\prod_{i\in I} |\sigma_i|^{2e_i}}{\prod_{j\in J} |s_j|^{2b_j}}\omega^n
\end{equation}
on $\cX_t$. Here $\lambda\geq 0$ is a positive real, and the parameters
$e_i, b_j$ are chosen as above. In case we have $\lambda= 0$, the normalization we choose for the solution is
\begin{equation}\label{eq20}
\int_{\cX_t}\varphi_t\omega_{\varphi_t}^n= 0.
\end{equation}
The function $f$ in \eqref{eq19} is supposed to be smooth on the total space $\cX$.
\smallskip

\noindent We consider the family of approximations of \eqref{eq19}
\begin{equation}\label{eq23}
  (\omega+ dd^c\varphi_\ep)^n= e^{\lambda \varphi_\ep+ f}\frac{\prod_{i\in I}
    (\ep^2 +|\sigma_i|^{2})^{e_i}}{\prod_{j\in J}(\ep^2+ |s_j|^2)^{b_j}}\omega^n
\end{equation}
on $\cX_t$. By general results in MA theory, the function $\varphi_\ep$ obtained by glueing the
fiberwise solutions of \eqref{eq23} is smooth. We will analyze in the next subsections the
uniformity with respect to $\ep$ of several norms of $\varphi_\ep$.
\medskip

\noindent We recall the following important result whose origins can be found in \cite{Yau78}.

\begin{theo}\label{Yau}
 For any strictly smaller disk $\DD^\prime\subset \DD$
there exists a constant $C> 0$ such that we have
\begin{equation}\label{eq21}
\Vert \varphi_\ep\Vert_{\cC^1(\cX_t)}\leq C
\end{equation}
for all $t\in \DD^\prime$, where the $\cC^1$ norm above is with respect to a fixed metric
which is quasi-isometric to \eqref{wsp4}.
\end{theo}

\noindent If $b_j=0$, this is a consequence of \cite{Yau78}, cf. also the version established in \cite{Paun}, stating that
\begin{equation}\label{eq22}
\omega+ dd^c\varphi_\ep\leq C\omega|_{\cX_t}.
\end{equation}
The conic case is much more involved and we refer to Theorem~\ref{thm:conic_grad} and the few lines following that statement, on page~\pageref{pageyau}. Note that inequality \eqref{eq22} is still true provided that we replace
the RHS with $\displaystyle C\omega_{B, \ep}|_{\cX_t}$, where $\omega_{B, \ep}$ is the regularization of a
conic metric corresponding to $(X, B)$ which is quasi-isometric with \eqref{wsp4}.
\smallskip

\noindent During the rest of the current subsection we assume that $\lambda= 0$, which is anyway what we need for the proof of Theorem \ref{thmc}. We will explain along the way how
to adapt our method to the case $\lambda> 0$.

\subsubsection{Mean value of the $t$-derivative} Let $v$ be a smooth vector field on $\cX$ of (1,0)-type,
which has the following properties.
\begin{enumerate}

\item[$(i)$] It is a lifting of $\displaystyle \frac{\partial}{\partial t}$, i.e. we have
\begin{equation}
dp(v)= \frac{\partial}{\partial t}
\end{equation}
(with the usual abuse of notation...).

\item[$(ii)$] We write $v$ locally as in \eqref{eq1}; then on $\Omega_j$ we have
\begin{equation}
|v^\alpha(z_j)|\leq C|z_j^\alpha|
\end{equation}
(we use the notations/conventions as in \eqref{E1}) for all $\alpha= 1,\dots , d$. This means
that $v$ is a smooth section of the logarithmic tangent space of $(X, E_{\rm red}+ B_{\rm red})$.
\end{enumerate}

\noindent Such a vector field $v$ is easy to construct, by a partition of unit of local lifts of
$\displaystyle \frac{\partial}{\partial t}$. We consider the coordinate sets
$\Omega_j$ and the $z_j$ adapted to the pair $(\cX, B+ E)$. Then the particular form of the transition  implies (ii).
\medskip

\noindent In this context we have the following statement.

\begin{lemm}\label{MV}
 There exists a constant $C> 0$ independent of $\ep$ such that we have
\begin{equation}\label{wsp12}
\left|\int_{\cX_t}v(\varphi_\ep)\omega_{\varphi_{\ep}}^n\right|\leq C
\end{equation}
for any $t\in \DD^\prime$.
\end{lemm}
\begin{proof}
  We consider a covering of $\cX$ by coordinate sets
  $\displaystyle \left(U_i, (z_i, t)\right)_{i}$ where the last coordinate $t$ is given by the map $p$. The normalization condition $\eqref{eq20}$ can be written as
\begin{equation}\label{wsp10}
\sum_i\int_{\Vert z_i\Vert< 1}\theta_i(z_i, t) \varphi_\ep(z_i, t)\frac{\prod_{\alpha \in I}\left(\ep^2+ |z_i^\alpha|^2e^{\phi_\alpha(z_i, t)}\right)^{e_\alpha}}{\prod_{\beta\in J}\left(\ep^2+ |z_i^\beta|^2e^{\psi_\beta(z_i, t)}\right)^{b_\beta}}e^{F_i(z_i, t)}d\lambda(z_i)
\end{equation}
where $\theta_i$ is a partition of unit, $I\cap J= \emptyset$ and $e^{F_i(z_i, t)}d\lambda(z_i)$ is the volume element $\omega^n$ restricted to $\cX_t$.
We take the $t$-derivative of \eqref{wsp10} and we have
\begin{equation}\label{wsp11}
\sum_i\int_{\Vert z_i\Vert< 1}\theta_i(z_i, t) \frac{\partial \varphi_\ep(z_i, t)}{\partial t}\frac{\prod_{\alpha \in I}\left(\ep^2+ |z_i^\alpha|^2e^{\phi_\alpha(z_i, t)}\right)^{e_\alpha}}{\prod_{\beta\in J}\left(\ep^2+ |z_i^\beta|^2e^{\psi_\beta(z_i, t)}\right)^{b_\beta}}e^{F_i(z_i, t)}d\lambda(z_i)= O(1)
\end{equation}
where $\O(1)$ above is uniform with respect to $t, \ep$ by the
$\cC^0$ estimates for $\varphi_\ep$.
Now by the construction of the vector $v$ above
the LHS of \eqref{wsp11} is precisely \eqref{wsp12}, so the lemma follows.
\end{proof}

\subsubsection{$L^2$-bound of the $t$-derivative}

We rewrite the relation corresponding to \eqref{eq18} in our setting; during the
next computations, we denote by
\begin{equation}
\tau:= v(\varphi_\ep)
\end{equation}
and then we have
{\small
\begin{equation}\label{eq24}
\Delta_{\varphi_\ep} \tau - \Tr_{\varphi_\ep} \partial \big(\dbar v\cdot \varphi_\ep\big)+
\Psi_{\varphi_\ep, v}= \lambda \tau+ v(f)+
\sum_je_jv\big(\log(\ep^2+ |\sigma_j|^2)\big)- \sum_ib_i v\big(\log(\ep^2+ |s_i|^2)\big)+ \Psi_v.
\end{equation}}
\smallskip

\noindent The equality \eqref{eq24} will be used in order to establish the following
statement.
\begin{prop} \label{L2L1}
There exists a constant $C> 0$ such that we have
\begin{equation}\label{eq25}
\int_{\cX_t}|\nabla_\ep \tau|_\ep^2\,\omega_{\varphi_\ep}^n\leq
C\left(1+ \int_{\cX_t}|\tau|\,\omega_{\varphi_\ep}^n\right)
\end{equation}
for any $\ep> 0$. The operator $\nabla_\ep$ is the gradient corresponding to the
metric $\displaystyle \omega_{\varphi_\ep}$.
\end{prop}
\begin{proof}
In order to establish \eqref{eq25} we multiply the relation \eqref{eq24} with $\ol \tau$ and then we
integrate the result on $\cX_t$
against the measure $\omega_{\varphi_\ep}^n$. A few observations are in order.
\begin{enumerate}

\item[$\bullet$] We have
\begin{equation}
\sup_{\cX_t}\left(\left|v(f)\right|+ \left|\sum_{j}e_jv\big(\log(\ep^2+ |\sigma_j|^2)\big)\right|+ \left|\sum_{i}b_iv\big(\log(\ep^2+ |s_i|^2)\big)\right|+
\left|\Psi_v\right|\right)\leq C
\end{equation}
uniformly with respect to $\varepsilon$, by the property (ii) of the vector field $v$ and the definition
\eqref{eq17} of the function $\Psi_v$.

\item[$\bullet$] Since the constant $\lambda$ is positive, the $L^2$ norm of $\sqrt{\lambda}\tau$
will be on the left-hand side part of \eqref{eq25}, hence the presence of a strictly positive
$\lambda$ would
reinforce the inequality we want to obtain.
\end{enumerate}
\medskip

\noindent The terms
\begin{equation}\label{eq26}
\Tr_{\varphi_\ep} \partial \big(\dbar v\cdot \varphi_\ep\big), \quad \Psi_{\varphi_\ep, v}
\end{equation}
are kind of troublesome, because we do not have a $L^\infty$ bound for them. Nevertheless,
we recall that we only intend to establish an inequality between $L^p$ norms, and we will use
integration by parts to deal with \eqref{eq26}.

For the first term in \eqref{eq26} we argue as follows: integration by parts gives
\begin{equation}\label{eq5}
\int_{\cX_t}\ol \tau\partial \big(\dbar v\cdot \varphi_\ep\big)\wedge \omega_{\varphi_\ep}^{n-1}=
- \int_{\cX_t}\partial \ol \tau\wedge \dbar v\cdot \varphi_\ep\wedge \omega_{\varphi_\ep}^{n-1}
\end{equation}
and then we use Cauchy-Schwarz: the $L^2$ norm of $\dbar\tau$ is what we are after, but on the right hand side term we have it squared. The $L^2$ norm of $\dbar v\cdot \varphi_\ep$
is completely under control, because it only involves the fiber-direction derivatives of $\varphi_\ep$.

The second term is tamed in a similar manner. By definition of $\Psi_{\varphi_\ep, v}$ we have
\begin{equation}\label{eq27}
\int_{\cX_t}\ol\tau \Psi_{\varphi_\ep, v} \,  \omega_{\varphi_\ep}^{n}=
\int_{\cX_t}\ol\tau \cL_v(\omega)\wedge \omega_{\varphi_\ep}^{n-1}
\end{equation}
and by Cartan formula this is equal to
\begin{equation}\label{eq28}
\int_{\cX_t}\ol\tau d(i_v\cdot \omega)\wedge \omega_{\varphi_\ep}^{n-1}= \int_{\cX_t}\ol\tau
\partial(i_v\cdot \omega)\wedge \omega_{\varphi_\ep}^{n-1}.
\end{equation}
By Stokes formula the term \eqref{eq28} is equal to
\begin{equation}\label{eq29}
\int_{\cX_t}\partial\ol\tau \wedge
(i_v\cdot \omega)\wedge \omega_{\varphi_\ep}^{n-1}
\end{equation}
and now things are getting much better, in the sense that the $(0,1)$--form $i_v\cdot \omega$
is clearly smooth, so its $L^2$ norm with respect to $\omvpe$ is dominated by $C\int_{\cX_t} \om\wedge \omvpe^{n-1}\le C'$ and we use the Cauchy-Schwarz inequality.\\

\noindent All in all, we infer the existence of two constants $C_1$ and $C_2$ such that we have
\begin{equation}\label{eq30}
\int_{\cX_t}|\nabla_\ep \tau|^2\omega_{\varphi_\ep}^n\leq
C_1\int_{\cX_t}|\tau|\omega_{\varphi_\ep}^n+
C_2\left(\int_{\cX_t}|\nabla_\ep \tau|^2\omega_{\varphi_\ep}^n\right)^{1/2}
\end{equation}
for any positive $\ep> 0$.
The inequality \eqref{eq25} follows.
\end{proof}
\medskip

\noindent We infer the following statement.

\begin{theo}\label{L2N} There exists a positive integer $N\in \Z_+$ and a positive constant $C$
such that we have
\begin{equation}\label{eq31}
\int_{\cX_t}|\tau|^2d\mu^{(\ep)}_{Ne}\leq C
\end{equation}
for every positive $\ep$.
\end{theo}

\begin{proof} The arguments which will follow are absolutely standard, by combining the
Sobolev and Poincar\'e inequalities with \eqref{eq25}. Prior to this, we recall that we have
\begin{equation}\label{eq32}
\omega_{\ep}\leq C\omega_{B, \ep}
\end{equation}
on each $\cX_t$ for some constant $C$ which is uniform with respect to $\ep$
and with respect to $t\in \DD^\prime$. On the RHS of \eqref{eq32} we have
$\omega_{B, \ep}$ which stands for any metric quasi-isometric with
\eqref{wsp4}. In particular, for any function $f$ we have
\begin{equation}\label{eq33}
|\nabla f|\leq C|\nabla_\ep f|_\ep
\end{equation}
where the symbols $|\cdot |, \nabla$ and $ |\cdot |_\ep, \nabla_\ep$
correspond to the metric $\omega_{B, \ep}$ and $\omega_\ep$ respectively.

Now, Poincar\'e inequality \ref{WP} applied for $\alpha= 1$ combined with Lemma \ref{MV}
gives
\begin{equation}\label{eq34}
\int_{\cX_t}|\tau|d\mu^{(\ep)}_{e}\leq C\left(1+ \int_{\cX_t}|\nabla \tau|d\mu^{(\ep)}_{e/2}\right).
\end{equation}

On the other hand we have
\begin{eqnarray*}
\int_{\cX_t}|\nabla \tau|d\mu^{(\ep)}_{e/2} & \leq & C\int_{\cX_t}|\nabla_\ep \tau|_\ep d\mu^{(\ep)}_{e/2}
\\
& \le & C\left(\int_{\cX_t}|\nabla_\ep \tau|_\ep^2 d\mu^{(\ep)}_{e}\right)^{1/2}
\\
& \le & C+ C\left(\int_{\cX_t}|\tau|d\mu^{(\ep)}_{e}\right)^{1/2}
\end{eqnarray*}
where we have used Proposition \ref{L2L1} for the last inequality. When combined with \eqref{eq34},
this implies
\begin{equation}\label{eq35}
\int_{\cX_t}|\tau|d\mu^{(\ep)}_{e}\leq C
\end{equation}
for any $\ep> 0$.
\smallskip

\noindent We define next the sequence of rational numbers
\begin{equation}\label{eq36}
p_1=1, \quad p_{k+1}:= \frac{2np_k}{2n- p_k}
\end{equation}
as well as the sequence
\begin{equation}\label{eq37}
q_1= e, \quad q_{k+1}:= \frac{2}{2- p_k}q_k.
\end{equation}
One can actually find a closed formula for $p_k=\frac{2n}{2n-k+1}$ holding for $1\le k \le 2n$. It also follows that $p_k<2$ as long as $1 \le k \le n$ which is thus the range of integers for which $q_{k+1}$ is defined; one can also check the formula $q_{k+1}=\frac{(2n)!(n-k)!}{n!(2n-k)!}\cdotp q$. In particular $q_{n+1}=\frac{(2n)!}{n!^2} \cdotp q$. This is the factor $N$ in the statement of the proposition.

We observe that for $k=1,\dots n$ the components of $q_k$ are positive rational numbers, greater than
the respective components of $q$.

\noindent The Sobolev inequality \ref{WS} gives
\begin{equation}\label{eq39}
\left(\int_{\cX_t}|\tau|^{p_{k+1}}d\mu^{(\ep)}_{q_{k+1}}\right)^{\frac{1}{p_{k+1}}}\leq
C\left(\int_X|\nabla_\ep \tau|_\ep^{p_k}d\mu^{(\ep)}_{q_k} +  \int_{\cX_t}|\tau|^{p_k}d\mu^{(\ep)}_{q_k} \right)^{\frac{1}{p_k}}
\end{equation}

\noindent We iterate \eqref{eq29} for $k=1,\dots n$ and the Proposition \ref{L2N} is proved by observing that
the following holds.

\noindent $\bullet$ We have
$\displaystyle \int_{\cX_t}|\nabla_\ep \tau|_\ep^2\omega_{\varphi_\ep}^n\leq C$, by Proposition \ref{L2L1}, combined with \eqref{eq35} and the fact that the quotient of the two measures
\begin{equation}
\omega_{\varphi_\ep}^n, d\mu^{(\ep)}_e
\end{equation}
is uniformly bounded both sides.
\smallskip

\noindent $\bullet$ For each $k= 1,\dots n$ we have
\begin{equation}
\int_{\cX_t}|\nabla_\ep \tau|_\ep^{p_k}d\mu^{(\ep)}_{q_k}\leq C
\left(\int_{\cX_t}|\nabla_\ep \tau|_\ep^{2}d\mu^{(\ep)}_{\frac{2}{p_k}q_k}\right)^{p_k/2}
\leq C
\end{equation}
where the first inequality is simply Cauchy-Schwarz, and the second one is due to the fact that
we have
\begin{equation}
d\mu^{(\ep)}_{\frac{2}{p_k}q_k}\leq C \omega_{\varphi_\ep}^n
\end{equation}
because $\frac{q_k}{p_k} \ge \frac q 2$. This last inequality follows by induction given that $\frac{q_{k+1}}{p_{k+1}}=\frac{2n-p_k}{2n-np_k}\cdotp \frac{q_k}{p_k} $.
\end{proof}
\medskip

\subsection{A gradient estimate in the conic case}
 \begin{theo}
\label{thm:conic_grad}
Let $(X,\om)$ be a compact Kähler manifold, and let $\omvp:=\om+\ddc \vp$ be a Kähler metric satisfying $$\omvp^n= e^{\lambda \vp +F} \om^n$$ for some $F\in \mathscr C^{\infty}(X) $ and $\lambda\in \R$. We assume that there exists $C>0$ and a smooth function $\Psi, \Phi$ such that:
\begin{enumerate}
\item[$(i)$] $\sup_X |\vp| \le C$
\item[$(ii)$] $ \sup_X |\Psi|\le C$ and for any $\delta>0$, there exists $C_{\delta}$ such that
\begin{enumerate}
\item[$a.$] $\ddc \Psi \ge \delta^{-1}d\Psi \wedge d^c \Psi-C_{\delta} \om$
\item[$b.$] $\Delta_{\om} \Psi \ge \delta^{-1} |\nabla F|_{\om}-C_{\delta}$
\end{enumerate}
\item[$(iii)$] $i\Theta_{\om}(T_X) \ge -(C\om + \ddc \Psi) \otimes \mathrm{Id}$
\item[$(iv)$] $\omvp \le C \om$
\end{enumerate}
Then there exists a constant $A>0$ depending only on $C$ and $n$ such that $|\nabla \vp |_{\om} \le C$.
\end{theo}

As a corollary of this result, the gradient estimate \eqref{eq21} in Theorem~\ref{Yau} holds.

\begin{proof}[Proof of Theorem~\ref{Yau}]
\label{pageyau}
Let us rewrite the equation \eqref{eq23}
as
$$(\om_{\ep}+\ddc u_{\ep})^n=e^{\lambda u_\ep+ f_{\ep}}{\prod_{i\in I}
    (\ep^2 +|\sigma_i|^{2})^{e_i}}\om_{\ep}^n$$
where the reference metric $\om_{\ep}\in \{\om\}$ is an approximate conic metric along the divisor $B$, and $u_{\ep}$ differs from $\vp_{\ep}$ by a function whose $L^{\infty}$ norm as well as gradient and complex Hessian are uniformly bounded with respect to $\om_{\ep}$. Therefore it is sufficient to establish \eqref{eq21} for $u_{\ep}$. We check successively that conditions $(i)-(iv)$ are satisfied.

The bound $(i)$ follows from Ko\l odziej's estimate. It is straighforward when $\lambda=0$, and when $\lambda>0$, it requires an additional step easily achieved with Jensen inequality. Next, we choose $\Psi_{\ep}:=C(\sum _i (|\sigma_i|^{2}+\ep^2)^{\rho}+ \sum_j ( |s_j|^2+\ep^2)^{\rho})$ for $C$ large enough and $\rho>0$ small enough. Condition $(ii).a$ can be checked independently for each summand of $\Psi_{\ep}^{\alpha}$ of $\Psi_{\ep}$ in which case if follows from the fact that $\Psi_{\ep}^{\alpha}$ is uniformly quasi-psh (hence $C\om_{\ep}$-psh). Condition $(ii).b$ is an easy computation combined with \cite[Sect.~5.2]{GP}. Condition $(iii)$ is showed in \cite[Sect.~4]{GP}, while $(iv)$ is the content of \cite[Prop.~1]{GP}. To be more precise, \textit{op. cit.} assumes an upper and lower bound on $f_{\ep}+\sum e_i \log(|\sigma_i|^2+\ep^2)$ in order to get a two-sided inequality for $\om_{\vp}$; however one only needs an upper bound for the previous quantity if one only wishes to prove the one-sided inequality $(iv)$.
\end{proof}

\begin{proof}[Proof of Theorem~\ref{thm:conic_grad}]
Let $\beta:= |\nabla \vp|^2$ (computed with respect to $\om$) and $\alpha:= \log \beta - \gamma \circ \vp$ where $\gamma $ is a function to specify later. Without loss of generality, one can assume $\inf \vp =0$, and we set $\sup \vp =: C_0$. We use the local notation $(g_{i\bar j})$ for $\om$. We work at a point $y\in X$ where $\alpha+2\Psi$ attains its maximum, and we choose a system of geodesic coordinates for $\om$ such that $g_{i\bar j}(y)=\delta_{i \bar j}$, $dg_{i\bar j}(y)=0$, and $\vp_{i\bar j}$ is diagonal.
We set $u_{i\bar j}=g_{i\bar j}+\vp_{i\bar j}$ the components of the metric $\omvp$. As $\alpha_p= \frac{\beta_p}{\beta}-\gamma'\circ \vp \, \vp_p$ and $\alpha_p(y)=-2\Psi_p(y)$, one has
\begin{equation}
\label{beta}
\frac{\beta_p}{\beta}(y) =(\gamma'\circ \vp(y)) \vp_p(y)-2\Psi_p(y)
\end{equation}

\noindent
Moreover, some computations show that
$$\alpha_{p\bar p}=\frac{1}{\beta}\left(R_{j\bar k p \bar p} \vp_j \vp_{\bar k} +2 \mathrm{Re}\sum_ju_{p\bar p j}\vp_{\bar j}+\sum_j |\vp_{jp}|^2+\vp_{p\bar p}^2\right)-\frac{|\beta_p|^2}{\beta^2}-2\lambda-\gamma'' |\vp_p|^2-\gamma' \vp_{p\bar p}$$
Therefore at $y$, one gets from \eqref{beta} the following inequality:

\begin{eqnarray}
\label{eq:grad}
\alpha_{p\bar p}&\ge&\frac{1}{\beta}\left(R_{j\bar k p \bar p} \vp_j \vp_{\bar k} +2 \mathrm{Re}\sum_ju_{p\bar p j}\vp_{\bar j}+\sum_j |\vp_{jp}|^2+\vp_{p\bar p}^2\right)-\\
&&2\lambda-\gamma''|\vp_p|^2-\gamma' \vp_{p\bar p}-\big|\gamma' \vp_p-2\Psi_p\big|^2 \nonumber
\end{eqnarray}
so at $y$, the RHS is non-positive.\\

\noindent
\textit{Step 1. The curvature term}

\noindent
By the assumption $(iii)$, we have for all $a,b$: $R_{j\bar k p\bar q} a_j \bar a_{ k} b_p \bar b_{ q} \ge -(C|a_j|^2+\Psi_{j \bar k}a_j \bar a_{k})|b|^2$ and by symmetry of the curvature tensor, we get $R_{j \bar k p \bar q}a_j \bar a_{ k} b_p \bar b_{ q} \ge -(C|b_p|^2+\Psi_{p \bar q}b_p \bar b_{ q})|a|^2$. We apply that to $a= \nabla \vp$ and $b$ the vector with only non-zero component the $p$-th one, equal to $\sqrt{u^{p\bar p}}$, we get:
$u^{p\bar p}R_{j \bar k p \bar p}\vp_k \vp_{\bar l} \ge -(Cu^{p\bar p}+u^{p\bar p}\Psi_{p \bar p})|\nabla \vp|^2$. As a consequence,
\begin{equation}
\label{courbure}
\frac{1}{\beta} \sum u^{p \bar p} R_{j \bar k p \bar p}\vp_j \vp_{\bar k} \ge -C\sum u^{p\bar p}-\sum_p u^{p\bar p}\Psi_{p\bar p}
\end{equation}
Therefore, Equation \eqref{eq:grad} becomes, at $y\in X$:
\begin{eqnarray}
\Delta' (\alpha+\Psi) &\ge&(\gamma'-C) \tr_{\omvp}\om +\frac{1}{\beta}\sum_p u^{p\bar p}\left(2 \mathrm{Re}\sum_ju_{p\bar p j}\vp_{\bar j}+\sum_j |\vp_{jp}|^2\right) - \nonumber\\
&&\gamma''|\nabla^{\om} \vp |^2_{\omvp}-n \gamma'-\sum_p u^{p\bar p}\big|\gamma' \vp_p-2\Psi_p\big|^2 -C \label{eq:grad2} \\ \nonumber
\end{eqnarray}

\noindent
\textit{Step 2. The gradient term}

\noindent
The next term to analyze is
\begin{equation}
\label{gradient}
\frac{1}{\beta} \sum_p u^{p\bar p}\left( 2 \mathrm{Re}\sum_ju_{p\bar p j}\vp_{\bar j}\right)=\frac{2}{\beta}\mathrm{Re}\sum_j F_j \vp_{\bar j}
\end{equation}
by \cite[1.13]{Blocki09}, and this term is dominated (in norm) by $2 |\nabla F| \beta^{-1/2}$ and at the point $y$, $\beta$ can always be assumed to be larger than $1$ so that our term is bigger that $-2 |\nabla F|$. In particular, one gets at $y$:
\begin{eqnarray}
\label{eq:grad3}
\Delta' (\alpha+\Psi) &\ge&(\gamma'-C) \tr_{\omvp}\om +\sum_p u^{p\bar p}\left(\frac{1}{\beta}\sum_j |\vp_{jp}|^2-\big|\gamma' \vp_p-2\Psi_p\big|^2\right) -\\
&&\gamma''|\nabla^{\om} \vp |^2_{\omvp}-n \gamma'-2|\nabla F| -C \nonumber \\ \nonumber
\end{eqnarray}

\noindent
\textit{Step 3. Using the second derivatives}

%
\noindent
Recall that $\beta_p= \sum_j \vp_{jp}\vp_{\bar j}+\vp_p(u_{p\bar p}-1)$.
At $y$, $\frac{\beta_p}{\beta}-\gamma'  \vp_p=-2 \Psi_p$ so that at this point, one has
$$ \sum_j \vp_{jp}\vp_{\bar j}=( \gamma' \beta +1-u_{p\bar p})\vp_p-2\beta \Psi_p$$
hence $\left|\sum_j \vp_{jp}\vp_{\bar j}\right| = \beta\left|(\gamma'\vp_p-2\Psi_p)+\beta^{-1}(1-u_{p\bar p})\vp_p\right|$.
By Schwarz inequality, $\left|\sum_j \vp_{jp}\vp_{\bar j}\right|^2 \le \beta \sum_j |\vp_{jp}|^2$ and therefore
\begin{eqnarray*}
\label{eq:grad4}
\frac{1}{\beta}\sum_j |\vp_{jp}|^2-\big|\gamma' \vp_p-2\Psi_p\big|^2 &\ge& \left|(\gamma'\vp_p-2\Psi_p)+\beta^{-1}(1-u_{p\bar p})\vp_p\right|^2-\big|\gamma' \vp_p-2\Psi_p\big|^2 \\
&\ge & -2\beta^{-1}|1-u_{p\bar p}|\cdotp\big|\gamma' \vp_p-2\Psi_p\big|\cdotp |\vp_p|
\end{eqnarray*}
and by $(iv)$, $|1-u_{p\bar p}| \le C$, so that we get:
$$\sum_p u^{p\bar p} \left(\frac{1}{\beta}\sum_j |\vp_{jp}|^2-\big|\gamma' \vp_p-2\Psi_p\big|^2\right) \ge -C( \tr_{\omvp}\om+|\nabla \Psi|^2_{\omvp})$$
Combining this last inequality with \eqref{eq:grad3}, we get at $y$:
\begin{eqnarray*}
\label{eq:grad5}
0 &\ge& \Delta' (\alpha+2\Psi) \\
&\ge&(\gamma'-C) \tr_{\omvp}\om-\gamma''|\nabla^{\om} \vp |^2_{\omvp}-n \gamma'+\left( \Delta' \Psi-C|\nabla \Psi|^2_{\omvp}-2|\nabla F|_{\om}\right)-C
\end{eqnarray*}
As $\Psi$ is quasi-psh and $\omvp\le C\om$, we have $\Delta' \Psi\ge C^{-1} \Delta \Psi -C\tr_{\omvp}\om$ so by $(ii).b$, $\Delta' \Psi \ge 4 |\nabla F|_{\om}-C(1+\tr_{\omvp}\om)$. Using $(ii).a$, one ends up with the following inequality at $y$:
$$(\gamma'-C) \tr_{\omvp}\om-\gamma''|\nabla^{\om} \vp |^2_{\omvp}-n \gamma'\le C$$
Choosing $\gamma(t)=(C+1)t-||\vp||_{\infty}^{-1}t^2$ enables to conclude just as in \cite{Blocki09}.
\end{proof}

\begin{proof}[Proof of Theorem \ref{thmc}]
\label{thme}
It is a combination of our preceding considerations. The equation which gives $\omega_{\rm KE}$ fiberwise is
of the same type as \eqref{eq19} (with $\lambda= 0$). We conclude by Theorem \ref{Yau} and Theorem \ref{L2N}.
\end{proof}
\section{Existence of non-semipositive relative Ricci-flat Kähler metrics}
\label{nonpositive}
Let $p:X\to Y$ be a holomorphic fibration between projective manifolds of relative dimension $n\ge 1$. Let $\Yz$ be the set of regular values, and let $\Xz:=p^{-1}(\Yz)$. We assume that for $y\in \Yz$, $c_1(K_{X_y})=0$, where $X_y:=p^{-1}(y)$. Let $L$ be a pseudoeffective, $p$-ample $\mathbb Q$-line bundle on $X$.  One can write $L=H+p^*M$ for some ample line bundle $H$ on $X$ and for some line bundle $M$ on $Y$. In particular, one can find a smooth $(1,1)$-form $\om\in c_1(L)$ on $X$ such that for any $y\in Y°$, $\om_y:=\om|_{X_y}$ is a Kähler form on $X_y$.

\noindent
By Yau's theorem, there exists for any $y\in Y°$ a unique function $\vp_y \in \mathcal C^{\infty}(X_y)$ such that:
\begin{enumerate}
\item[$(i)$] $\theta_y:=\om_y+\ddc \vp_y$ is a Kähler form
\item[$(ii)$]  $\int_{X_y} \vp_y \om_y^n=0$
\item[$(iii)$]  $\Ric \theta_y = -\ddc \log \om_y^n=0$
\end{enumerate}
Moreover, one can use the implicit function theorem to check that the dependence of $\vp_y$ in $y$ is smooth, so that the form $\theta:=\om+\ddc \vp$ is a well-defined smooth $(1,1)$-form on $X°$ which is relatively Kähler. It is a folklore conjecture that the form $\theta$ is semipositive on $X$, say when $L$ is globally ample. Building on the results in the Appendix on page~\pageref{appendix}, we are able disprove this conjecture.

\begin{theo}
\label{counterexample}
There exists a projective fibration $p:X\to Y$ as in the setting above and an ample line bundle $L$ on $X$ such that the relative Ricci-flat metric $\theta$ on $X°$ associated with $L$ is not semipositive.
\end{theo}

\begin{rema}
The counter-example is actually pretty explicit: $X$ is a $K3$ surface and $p$ is an elliptic fibration onto $Y=\P^1$.
\end{rema}

\begin{proof}[Proof of Theorem~\ref{counterexample}]
We proceed in three steps, arguing by contradiction. That is, we assume that the folklore conjecture recalled above is true for any such fibration $p:X\to Y$. \\

\noindent
\textbf{Step 1. Choice of the fibration.}

\noindent
We consider a K3 surface $X$ provided by Proposition~\ref{exist}. Its (singular) fibers are irreducible and reduced. Moreover, $X$ admits a semi-ample line bundle $L$ which is $p$-ample and has numerical dimension one. Indeed, $L$ can be chosen as the pull-back of $\mathcal O_{\mathbb P^1}(1)$ by  another elliptic fibration $q:X\to \mathbb P^1$.  Moreover, one knows that $p$ is not isotrivial, in the sense that two general fibers $X_y$, $X_{y'}$ of $p$ are not isomorphic.  \\

\noindent
\textbf{Step 2. Reduction to the semi-ample case.}

\noindent
 Let us pick $A$ an ample line bundle on $X$, $\om_A \in c_1(A)$ a Kähler form, and let us consider the relative Ricci-flat form $\theta_{\ep}$ on $X°$ associated with the the pair $(L+\ep A, \om+\ep \om_A)$. The line bundle $L_{\ep}$ is ample, hence it follows from our assumption that for any $\ep>0$, the relative Ricci-flat metric satisfies $$\theta_{\ep} \ge 0 \quad \mbox{on } X°.$$ We are going to show that $\theta_{\ep}$ converges weakly on $X°$ to the current $\theta:=\theta_0$. As a result, this will force $\theta$ to be semipositive on $X°$.

Let us write $\theta_{\ep}=\om+\ep\om_A+\ddc \vp_{\ep}$ where $\vp_{\ep}$ is normalized such that for each $y\in Y°$, one has $$\int_{X_y}\vp_{\ep} (\om+\ep\om_A)=0.$$ If $C_{\ep}$ is the constant (converging to $0$) defined by $$e^{C_{\ep}}= \frac{[X_y] \cdotp c_1(L)}{[X_y] \cdotp c_1(L+\ep A)}$$ for any $y\in Y°$, then one has on $X_y$ the following equation:
$$ \om+\ep\om_A+\ddc \vp_{\ep} = e^{C_{\ep}} \cdotp (\om+\ddc \vp)$$
The family of potentials $(\vp_{\ep}|_{X_y})_{\ep, y}$ is normalized in a smooth way with respect to $\ep$ and $y$, and satisfies linear equations depending smoothly on the parameters as well. It is not difficult to see that the standard estimates hold uniformly in $\ep$ and $y$ (as long as $y$ evolves in compact subsets of $Y°$), hence uniqueness imposes that $\vp_{\ep} \to \vp$ smoothly in each $X_y$, locally uniformly in $y\in Y°$. In particular, $\vp_{\ep} $ converges weakly to $\vp$ in $L^1_{\rm loc}(X°)$. \\

\noindent
\textbf{Step 3. End of the proof.}

\noindent
Thanks to Step 2, the relative Ricci-flat metric $\theta=\om+\ddc \vp$ is semipositive on $X°$. Moreover, it follows from Proposition~\ref{greenA} that $\vp$ is bounded above near $X\setminus X°$, hence $\theta$ extends to a semi-positive current $\theta \in c_1(L)$ on the whole $X$.
Let $\mathcal F \subset T_X$ be the holomorphic foliation induced by the fibration $q : X\rightarrow \mathbb P^1$.
As the semi-positive current $\theta$ is in the class of $c_1 (L)= q^* (c_1 (\mathcal O_{\mathbb P^1}(1)))$ and $q$ has connected fiber, it follows that there exists a positive current $\gamma \in c_1 (\mathcal O_{\mathbb P^1}(1))$ such that $\theta=q^*\gamma$. In particular, if $X^1\subset X$ denotes the locus where $q$ is smooth and if $\Omega:=X°\cap X^1$, then $\mathcal F |_{\Omega}$ is contained in
the kernel $\mathrm{Ker} \, \theta$ on $\Omega$. As both foliations are smooth and have rank one on $\Omega$, one has
\begin{equation}
\label{egalite}
\mathcal F |_{\Omega}=\mathrm{Ker} \, \theta|_{\Omega}.
\end{equation}

Next, let us pick a trivializing open set $U\simeq \Delta \subset Y°$, and let $V\in \mathcal C^{\infty}(X°,  T_X^{1,0})$ be the lift of $\frac{\partial}{\partial t}$ with respect to $\theta$ over $U$, cf e.g. \cite[Sect.~1.1]{Gue16}. One knows that in a trivializing chart $(z,t)$ defined on a subset of $p^{-1}(U)$ such that $p(z,t)=t$, the vector field $V$ can be written as $$V=\frac{\partial}{\partial t}+ a(z,t) \frac{\partial}{\partial z}$$ for some smooth function $a$. The function $c:=\theta(V,V)$ satisfies the identity $\theta^2= c \, \theta \wedge idt\wedge d\bar t$, hence it vanishes identically on $p^{-1}(U)$, that is, $$V\in \mathcal C^{\infty}(p^{-1}(U),\mathrm{Ker} \, \theta).$$ Thanks to \eqref{egalite}, this shows that for any $x\in p^{-1}(U)\cap \Omega$, one has $\mathbb C \cdotp V(x) = \mathcal F_x$. In particular, there exists a non-vanishing, smooth function $f$ on $p^{-1}(U)\cap \Omega$ such that $fV$ is holomorphic on $p^{-1}(U)\cap \Omega$. Now in local coordinates, this means that
$$0=\bar \partial (fV)= \bar \partial f \otimes \frac{\partial}{\partial t}+ \bar \partial (fa) \otimes \frac{\partial }{\partial z}$$
hence $\bar \partial f=0$. As a result, the smooth vector field $V$ on $p^{-1}(U)$ is holomorphic on $p^{-1}(U)\cap \Omega$, hence on the whole $p^{-1}(U)$. Therefore, its flow induces a local biholomorphism between any two near fibers. In particular, any two smooth fibers over $U$ would be isomorphic, which contradicts the non-isotriviality of $p$.
\end{proof}

\newpage

\pagestyle{empty}
\addtocontents{toc}{\protect\setcounter{tocdepth}{0}}
\section*{Appendix by Valentino Tosatti\protect\footnote[2]{Department of Mathematics, Northwestern University, Evanston, IL 60208, USA \newline \emph{email}: tosatti@math.northwestern.edu \newline V.T was partially supported by NSF grant DMS-1610278 and by a Chaire Poincar\'e at Institut Henri Poincar\'e funded by the Clay Mathematics Institute. The author is grateful to Matthias Sch\"utt and Chenyang Xu for discussions, and to the Institut Henri Poincar\'e for the gracious hospitality.}}
\addtocontents{toc}{\protect\setcounter{tocdepth}{1}}
\label{appendix}
\addcontentsline{toc}{section}{Appendix by Valentino Tosatti}
\bigskip

Let $(X^{n},\omega_X)$ be a compact K\"ahler manifold, $Y$ a compact Riemann surface, and $f:X\to Y$ a surjective holomorphic map with connected fibers. Let $Y^0$ be the locus of regular values for $f$, whose complement in $Y$ is a finite set, and $X^0=f^{-1}(Y^0)$, which is Zariski open in $X$, so that $f:X^0\to Y^0$ is a proper holomorphic submersion. We will call the fibers over points in $Y\backslash Y^0$ the singular fibers of $f$.

Suppose that for every $y\in Y^0$ we have a smooth function $\rho_y$ on the fiber $X_y=f^{-1}(y)$ which satisfies
\begin{equation}\label{cond}
\omega_X|_{X_y}+\sqrt{-1}\ddbar\rho_y\geq 0,\quad \int_{X_y}\rho_y (\omega_X|_{X_y})^n=0.
\end{equation}

\begin{app}\label{greenA}
If all the singular fibers of $f$ are reduced and irreducible, then there is a constant $C$ such that
$$\sup_{X_y}\rho_y\leq C,$$
holds for all $y\in Y^0$.
\end{app}

\begin{proof}
Let $\omega_y=\omega_X|_{X_y}$, and $g_y$ be its Riemannian metric, where in the following we fix any $y\in Y^0$.
Thanks to \eqref{cond}, on $X_y$ we have
\begin{equation}\label{d1}
\Delta_{g_y}\rho_y\geq -n+1.
\end{equation}
We have that $\mathrm{Vol}(X_y,g_y)=c$, a constant independent of $y$, and that the Sobolev constant of $(X_y,g_y)$ has a uniform upper bound independent of $y$ thanks to the Michael-Simon Sobolev inequality \cite{MS}, see the details e.g. in \cite[Lemma 3.2]{Tos10}. Furthermore, $\mathrm{diam}(X_y,g_y)\leq C$, a constant independent of $y$, thanks to \cite[Lemma 3.3]{Tos10}.

So far we have not used the assumptions that all singular fibers are reduced and irreducible. This is used now to prove that the Poincar\'e constant of $(X_y,g_y)$ also has a uniform upper bound independent of $y$, as shown by Yoshikawa \cite{Yoshi} (see also the much clearer exposition in \cite[Proposition 3.2]{RZ0}).

At this point we can use a classical argument of Cheng-Li \cite{ChengLi}, which is clearly explained in \cite[Chapter 3, Appendix A, pp.137-140]{Siu87}, to deduce that the Green's function $G_y(x,x')$ of $(X_y,g_y)$, normalized by
$$\int_{X_y}G_y(x,x') \omega_y(x')=0,$$ satisfies the bound
\begin{equation}\label{d2}
G_y(x,x')\geq -A,
\end{equation}
for all $y\in Y^0$ and for all $x,x'\in X_y$, with a uniform constant $A$. The point of that argument is that $A$ only depends on the constant in the Sobolev-Poincar\'e inequality, that here as we said we control uniformly, on the dimension and on bounds for the volume and diameter, which we all have.

We can now apply Green's formula on $X_y$. Choose a point $x\in X_y$ such that $\rho_y(x)=\sup_{X_y}\rho_y$, and then, using that $\rho_y$ has average zero, together with \eqref{d1} and \eqref{d2}, we obtain
\[\begin{split}
\rho_y(x)&=-\int_{X_y}\Delta_{g_y}\rho_y(x') G_y(x,x')\omega_y(x')\\
&=-\int_{X_y}\Delta_{g_y}\rho_y(x') (G_y(x,x')+A)\omega_y(x')\\
&\leq (n-1)\int_{X_y}(G_y(x,x')+A)\omega_y(x')\\
&\leq (n-1)A\mathrm{Vol}(X_y,g_y).
\end{split}\]
\end{proof}

We now specialize to the setting where $X$ is a $K3$ surface, $Y=\mathbb{P}^1$ and $f:X\to \mathbb{P}^1$ is an elliptic fibration. We further assume that $\rho_y$ is chosen so that
$\omega_X|_{X_y}+\sqrt{-1}\ddbar\rho_y>0$ is the unique flat metric on $X_y$ cohomologous to $\omega_X|_{X_y}$ (and we still assume that $\rho_y$ has fiberwise average zero). In this case
$\rho_y$ varies smoothly in $y\in Y^0$, and so it defines a smooth function $\rho$ on $X^0$. Thanks to Proposition \ref{green}, we conclude that
$$\sup_{X^0}\rho\leq C.$$
This, together with the Grauert-Remmert extension theorem, immediately gives:
\begin{appcoro}
In this setting, if we have that $\omega_X+\sqrt{-1}\ddbar\rho\geq 0$ on $X^0$, then this extends to a closed positive current on all of $X$, in the class $[\omega_X]$.
\end{appcoro}

Lastly, we need the following examples:

\begin{app}\label{exist}
There exists a complex projective $K3$ surface $X$ which admits two elliptic fibrations, one of which is non-isotrivial and has only reduced and irreducible singular fibers.
\end{app}
\begin{proof}
Let $X\subset\mathbb{P}^2\times\mathbb{P}^1$ be a general hypersurface of degree $(3,2)$. It is known that $X$ has Picard number $2$ \cite[Section 5.8]{vG}. The projection to the $\mathbb{P}^1$ factor gives an elliptic fibration on $X$, which is clearly not isotrivial provided $X$ is general.

To obtain the other fibration we compose the first fibration with the automorphism $\sigma$ of $X$ obtained as follows. Projecting $X$ to the $\mathbb{P}^2$ factor shows that $X$ is a double cover of $\mathbb{P}^2$ ramified along a sextic, and the covering involution of this cover is the $\sigma$ that we want.

Explicitly, if we let $L=\mathcal{O}_{\mathbb{P}^2}(1)|_X$, $M=\mathcal{O}_{\mathbb{P}^1}(1)|_X$, the the first elliptic fibration is defined by $|M|$ and the second elliptic fibration by $|3L-M|$ (since $\sigma^*M=3L-M$).

Lastly, we show that every elliptic fibration on $X$ has only reduced and irreducible singular fibers. Given an elliptic fibration $f:X\to\mathbb{P}^1$, let $j:J\to\mathbb{P}^1$ be its Jacobian family \cite[Section 11.4]{HuyK3}. Then $J$ is also an elliptic $K3$ surface, every fiber of $j$ is isomorphic to the corresponding fiber of $f$, $J$ has the same Picard number as $X$, but $j$ always has a section. We can then apply the Shioda-Tate formula \cite[Corollary 11.3.4]{HuyK3} to $j$ to obtain
$$2=\rho(J)=2+\sum_{t\in\mathbb{P}^1}(r_t-1)+\mathrm{rank}\ \mathrm{MW}(j),$$
where $r_t$ is the number of irreducible components of the fiber $J_t$ and $\mathrm{MW}(j)$ is the Mordell-Weil group of $j$. In particular we conclude that $r_t=1$ for all $t$, i.e. all fibers of $j$ (and therefore all fibers of $f$) are irreducible. Lastly, all fibers of $f$ are reduced by \cite[Proposition 3.1.6 (iii)]{HuyK3}.
\end{proof}

\bibliographystyle{smfalpha}
\bibliography{biblio}

\end{document}